\documentclass[12pt, reqno]{article}
\usepackage{amsfonts}
\usepackage{amsmath}
\usepackage{amssymb}
\usepackage{latexsym,bm}
\usepackage{amsthm}
\usepackage{graphicx}
\usepackage{cite}
\usepackage[colorlinks,urlcolor=blue]{hyperref}

\usepackage[top=3cm,bottom=3cm,left=3cm,right=3cm,includehead,includefoot]{geometry}
\usepackage{algorithm}
\usepackage{algorithmic}
\allowdisplaybreaks

\newtheorem{thm}{Theorem}[section]

\newtheorem{lem}{Lemma}[section]

\newtheorem{ass}{Assumption}[section]
\theoremstyle{definition}
\newtheorem{defn}{Definition}[section]
\theoremstyle{Condition}

\theoremstyle{remark}
\newtheorem{rmk}{Remark}[section]
\numberwithin{equation}{section}
\theoremstyle{example}

\numberwithin{equation}{section}

\begin{document}

\bigskip
\bigskip

\bigskip

\begin{center}

\textbf{\large A stochastic coordinate descent inertial primal-dual
algorithm  for large-scale composite optimization }

\end{center}

\begin{center}
Meng Wen $^{1,2}$, Yu-Chao Tang$^{3}$, Jigen Peng$^{1,2}$
\end{center}

\begin{center}
1. School of Mathematics and Statistics, Xi'an Jiaotong University,
Xi'an 710049, P.R. China \\
2. Beijing Center for Mathematics and Information Interdisciplinary
Sciences, Beijing, P.R. China

 3. Department of Mathematics, NanChang University, Nanchang
330031, P.R. China
\end{center}

\footnotetext{\hspace{-6mm}$^*$ Corresponding author.\\
E-mail address: wen5495688@163.com}

\bigskip

\noindent  \textbf{Abstract} In this paper we consider an inertial
primal-dual algorithm to compute the minimizations of the sum  of
two convex functions and the composition of another convex function
with a continuous linear operator. With the idea of coordinate
descent, we design a stochastic coordinate descent inertial
primal-dual splitting algorithm.  Moreover, in order to prove the
convergence of the proposed inertial  algorithm, we formulate first
the inertial version of the randomized Krasnosel'skii-Mann
iterations algorithm for approximating the set of fixed points of a
nonexpansive operator and investigate its convergence properties.
Then the convergence of stochastic coordinate descent inertial
primal-dual splitting algorithm is derived by applying the inertial
version of the randomized Krasnosel'skii-Mann iterations  to the
composition of the proximity operator. Finally, we give two
applications of our method. (1) In the case of stochastic minibatch
optimization, the algorithm can be applicated to split a composite
objective function into blocks, each of these blocks being processed
sequentially by the computer. (2) In the case of distributed
optimization, we consider a set of $N$ networked agents endowed with
private cost functions and seeking to find a consensus on the
minimizer of the aggregate cost. In that case, we obtain a
distributed iterative algorithm where isolated components of the
network are activated in an uncoordinated fashion and passing in an
asynchronous manner. Numerical results demonstrate the efficiency of
the method in the framework of large scale machine learning
applications. Generally speaking, our method converges faster than
existing methods, while keeping the computational cost of each
iteration basically unchanged.


\noindent \textbf{MR(2000) Subject Classification} 47H09, 90C25,

\section{Introduction}

The purpose of this paper is to designing and discussing an
efficient algorithmic framework with inertial version for minimizing
the following problem
$$\min_{x\in\mathcal{X}} f(x)+g(x)+ (h\circ D)(x),\eqno{(1.1)}$$
where $\mathcal{X}$ and $\mathcal{Y}$ are two finite-dimensional
Euclidean spaces, and $n=dim\mathcal {X}$, $m=dim\mathcal {Y}$,
$f,g\in\Gamma_{0}(\mathcal{X})$, $h\in\Gamma_{0}(\mathcal{Y}),$ $f$
is differentiable on $\mathcal{Y}$  and
$D:\mathcal{X}\rightarrow\mathcal{Y}$ a linear transform. Here and
in what follows, for a real Hilbert space $\tilde{H}$,
$\Gamma_{0}(\tilde{H})$ denotes the collection of all proper lower
semi-continuous convex functions from $\tilde{H}$ to
$(-\infty,+\infty]$. Despite its simplicity, when $g=0$ many
problems in image processing can be formulated in the form of (1.1).

\par
In this paper, the contributions of us are the following aspects:
\par
(I) We  provide a modification of the primal-dual algorithm to solve
the general Problem (1.1), which is inspired by the inertial
forward-backward splitting method[22].  We refer to our algorithm as
IADMM$^{+}$. When $\alpha_{k}=0$, the ADMM$^{+}$ algorithm
introduced by Bianchi [2] is a special case of our algorithm.  In
particular, we propose simple and easy to compute diagonal
preconditioners for which convergence of the algorithm is guaranteed
without the need to compute any step size parameters. we call this
algorithm as PADMM$^{+}$.

\par
(II) Based on the results of Bianchi [2] and Radu Ioan et al [9], we
introduce the idea of  inertial version  on randomized
krasnosel＊skii mann iterations. The form of Krasnosel'skii-Mann
iterations can be translated into fixed point iterations of a given
operator having a contraction-like property. Interestingly,
IADMM$^{+}$ is a special instances of the Inertial
Krasnosel'skii-Mann iterations. By the view of stochastic coordinate
descent, we know that at each iteration, the algorithm is only to
update  a random subset of coordinates. Although this leads to a
perturbed version of the initial Inertial Krasnosel'skii-Mann
iterations, but it can be proved to preserve the convergence
properties of the initial unperturbed version. Moreover, stochastic
coordinate descent has been used in the literature [11,23-24] for
proximal gradient algorithms. We believe that its application to the
broader class of Inertial  Krasnosel'skii-Mann algorithms can
potentially lead to various algorithms well suited to large-scale
optimization problems.

\par
(III) We use our views to large-scale optimization problems which
arises in signal processing and machine learning contexts. We prove
that the general idea of stochastic coordinate descent gives a
unified framework allowing to derive stochastic inertial algorithms
of different kinds. Furthermore, we give two application examples.
Firstly, we propose a new preconditioned stochastic  approximation
algorithm  by applying stochastic coordinate descent on the top of
PADMM$^{+}$. The algorithm is called as preconditioned stochastic
minibatch primal-dual splitting algorithm (PSMPDS). Secondly, we
introduce a random asynchronous distributed optimization methods
with preconditioning that we call as preconditioned distributed
asynchronous primal-dual splitting algorithm  (PDAPDS). The
algorithm can be used to efficiently solve an optimization problem
over a network of communicating agents.  The algorithms are
asynchronous in the sense that some components of the network are
allowed to wake up at random and perform local updates, while the
rest of the network stands still. No coordinator or global clock is
needed. The frequency of activation of the various network
components is likely to vary.

The rest of this paper is organized as follows. In the next section,
 we introduce some notations used throughout in the paper. In section 3, we devote
to introduce  IPDS and IADMM$^{+}$ algorithm,  and the relation
between them, we also show how the IADMM$^{+}$ includes ADMM$^{+}$
and the Forward-Backward algorithm as  special cases. In section 4,
we present the preconditioned primal-dual algorithm and give
conditions under which convergence of the algorithm is guaranteed.
In section 5, we provide our main result on the convergence of
 Inertial Krasnosel'skii-Mann algorithms with randomized coordinate
descent.
 In section 6, we propose a stochastic approximation
algorithm from the PADMM$^{+}$.  In section 7, we addresse the
problem of asynchronous distributed optimization. In the final
section, we show the numerical performance and efficiency of propose
algorithm through some examples in the context of large-scale
$l_{1}$-regularized logistic regression.

\section{Preliminaries }
Throughout the paper, we  denote by $\langle \cdot, \cdot\rangle$
the inner product on $\mathcal{X}$ and by $\|\cdot\|$ the norm on
$\mathcal{X}$.

\begin{ass}
The infimum of Problem (1.1) is attained. Moreover, the following
qualification condition holds
$$0\in ri(dom\, h-D\, dom\, g).$$
\end{ass}
The dual problem corresponding to the primal Problem (1.1) is
written
$$\min_{y\in\mathcal{Y}} (f+g)^{\ast}(-D^{\ast}y)+ h^{\ast}(y),$$
where $a^{\ast}$ denotes the Legendre-Fenchel transform of a
function $a$ and where $D^{\ast}$ is the adjoint of $D$. With the
Assumption 2.1, the classical Fenchel-Rockafellar duality theory
[3], [10] shows that
$$\min_{x\in\mathcal{X}} f(x)+g(x)+ (h\circ D)(x)-\min_{y\in\mathcal{Y}} (f+g)^{\ast}(-D^{\ast}y)+ h^{\ast}(y).\eqno{(2.1)}$$

\begin{defn}
 Let $f$ be  a real-valued convex function on
$\mathcal{X}$, the operator prox$_{f}$ is defined by
\begin{align*}
prox_{f}&:\mathcal{X}\rightarrow\mathcal{X}\\
& x\mapsto \arg \min_{y\in
\mathcal{X}}f(y)+\frac{1}{2}\|x-y\|_{2}^{2},
\end{align*}
called the proximity operator of $f$.

\end{defn}

\begin{defn}
Let $A$ be a closed convex set of $\mathcal{X}$. Then the indicator
function of $A$ is defined as
$$
\iota_{A}(x) = \left\{
\begin{array}{l}
0,\,\,\, \,\,if x\in A,\\
\infty,\,\,\, otherwise .
\end{array}
\right.
$$
\end{defn}
It can easy see the proximity operator of the indicator function in
a closed convex subset $A$ can be reduced a projection operator onto
this closed convex set $A$. That is,
$$prox_{\iota_{A}}=proj_{A},$$
where proj is the projection operator of $A$.

\begin{defn}
(Nonexpansive operators and firmly nonexpansive operators [3]). Let
$\mathcal{H}$ be a Euclidean space (we refer to [3] for an extension
to Hilbert spaces). An operator $T : \mathcal{H} \rightarrow
{\mathcal{H}}$ is nonexpansive if and only if it satisfies
$$\|Tx-Ty\|_{2}\leq\|x-y\|_{2}\,\,\, for\,\,all\,\,\, (x,y)\in \mathcal{H}^{2}.$$
$T$ is firmly nonexpansive if and only if it satisfies one of the
following equivalent conditions:
\par
(i)$\|Tx-Ty\|_{2}^{2}\leq\langle Tx-Ty,x-y\rangle$\,\,\,
for\,\,all\,\,\, $(x,y)\in \mathcal{H}^{2}$;
\par
(ii)$\|Tx-Ty\|_{2}^{2}=\|x-y\|_{2}^{2}-\|(I-T)x-(I-T)y\|_{2}^{2}$\,\,\,
for\,\,all\,\,\, $(x,y)\in \mathcal{H}^{2}$.
\par
It is easy to show from the above definitions that a firmly
nonexpansive operator $T$ is nonexpansive.
\end{defn}
\begin{defn}
 A mapping $T : \mathcal{H} \rightarrow \mathcal{H}$ is said to be an averaged mapping, if it can
be written as the average of the identity $I$ and a nonexpansive
mapping; that is,
$$T = (1-\alpha)I +\alpha S,\eqno{(2.2)}$$
where $\alpha$ is a number in ]0, 1[ and $S : \mathcal{H}
\rightarrow \mathcal{H}$ is nonexpansive. More precisely, when (2.2)
or the following inequality (2.3) holds, we say that $T$ is
$\alpha$-averaged.
$$\|Tx-Ty\|^{2}\leq \|x-y\|^{2}-\frac{(1-\alpha)}{\alpha}\|(I -T)x-(I -T)y\|^{2},\forall x,y\in\mathcal{H}.\eqno{(2.3)}$$

\end{defn}
A 1-averaged operator is said non-expansive. A $\frac{1}{ 2}$
-averaged operator is said firmly non-expansive.
\begin{defn}
A operator $B$ is said to be single-valued and cocoercive with
respect to a linear, selfadjoint and positive definite map $L$; that
is, for all $x,y \in\mathcal {H}$
$$\langle B(x)-B(y),x-y\rangle\geq\| B(x)-B(y)\|^{2}_{L^{-1}}\eqno{(2.4)}$$
\end{defn}
where, as usual, we denote $\|x\|^{2}_{L^{-1}} = \langle L^{-1}x,
x\rangle$ . Note that in the most simple case where $L = l$ Id, $l >
0$, the operator $B$ is $1/l$ co-coercive and hence $l$-Lipschitz.
However, we will later see that in some cases, it makes sense to
consider more general $L$.

 We refer the
readers to [3] for more details. Let $M : \mathcal{H} \rightarrow
\mathcal{H}$ be a set-valued operator. We denote by $ran(M) := \{v
\in\mathcal{H} : \exists u \in\mathcal{H}, v \in Mu\}$ the range of
$M$, by $ gra(M) := \{(u, v) \in \mathcal{H}^{2} : v \in Mu\}$ its
graph, and by $M ^{-1}$ its inverse; that is, the set-valued
operator with graph ${(v, u) \in \mathcal{H}^{2} : v \in Mu}$. We
define $ zer(M) := \{u \in \mathcal{H} : 0\in Mu\}$. $M$ is said to
be monotone if $\forall(u, u' ) \in \mathcal{H}^{2},\forall(v, v' )
\in Mu\times Mu'$, $\langle u-u' , v-v' \rangle\geq 0$ and maximally
monotone if there exists no monotone operator $M'$ such that $
gra(M) \subset gra(M') \neq gra(M)$.

The resolvent $(I + M)^{-1}$ of a maximally monotone operator $M :
\mathcal{H} \rightarrow \mathcal{H}$ is defined and single-valued on
$\mathcal{H}$ and firmly nonexpansive. The subdifferential $\partial
J$ of $J\in \Gamma_{0}(\mathcal{H})$ is maximally monotone and $(I
+\partial J)^{-1} = prox_{J}$ .

 Further, let us mention some classes of operators that are used in the paper. The
operator $A$ is said to be uniformly monotone if there exists an
increasing function $\phi_{A} : [0;+1) \rightarrow [0;+1]$ that
vanishes only at 0, and
$$\langle x-y,u-v\rangle\geq\phi_{A}( \|x-y\|), \forall(x,u),(y,v)\in gra(A).\eqno{(2.5)}$$
Prominent representatives of the class of uniformly monotone
operators are the strongly monotone operators. Let
 $\gamma> 0$ be arbitrary. We say that $A$ is
$\gamma$-strongly monotone, if $\langle x-y,u-v\rangle\geq\gamma
\|x-y\|^{2}$, for all $(x,u),(y,v)\in gra(A)$.

\begin{lem}
(Baillon-Haddad Theorem [3, Corollary 18.16]). Let $J : \mathcal {H}
\rightarrow \mathcal {R}$ be convex, differentiable on $\mathcal
{H}$ and such that $\pi\nabla J$ is nonexpansive, for some $\pi\in
]0,+\infty[$. Then $\nabla J$ is $\pi$-cocoercive; that
is,$\pi\nabla J$ is firmly nonexpansive.
\end{lem}
\begin{lem}
((Composition of averaged operators [4, Theorem 3]). Let
$\alpha_{1}\in ]0, 1[$, $\alpha_{2}\in  ]0, 1]$, $T_{1}\in\mathcal
{A}(\mathcal {H},\alpha_{1})$, and $T_{2}\in\mathcal {A}(\mathcal
{H},\alpha_{2})$. Then $T_{1}\circ T_{2}\in \mathcal {A}(\mathcal
{H},\alpha')$,
 where
$$\alpha':=\frac{\alpha_{1}+\alpha_{2}-2\alpha_{1}\alpha_{2}}{1-\alpha_{1}\alpha_{2}}.$$
\end{lem}

\begin{lem}
( [13]). Let $\tilde{H}$ be a real Hilbert space with inner product
$\langle
\cdot,  \cdot \rangle$ and norm $\|\cdot\|$ , then \\
$\forall x,y\in \tilde{H}, \forall \alpha\in[0,1], \|\alpha
x+(1-\alpha)y\|^{2}=\alpha\|x\|^{2}+(1-\alpha)\|y\|^{2}-\alpha(1-\alpha)\|x-y\|^{2}.$

\end{lem}
\begin{lem}
( see[14-16]). Let $(\varphi^{k})_{k\in\mathbb{N}}$;
$(\delta_{k})_{k\in\mathbb{N}}$ and $(\alpha_{k})_{k\in\mathbb{N}}$
be sequences in [0;+1) such that
$\varphi^{k+1}\leq\varphi^{k}+\alpha_{k}(\varphi^{k}-\varphi^{k-1})+\delta_{k}$
for all $k\geq 1$, $\sum_{k\in\mathbb{N}}\delta_{k} < +\infty$ and
there exists a real number $\alpha$ with $0\leq \alpha_{k}
\leq\alpha < 1$ for all $k\in\mathbb{N}$. Then the following hold:\\
(i) $\sum_{k\geq1}[\varphi^{k}-\varphi^{k-1}]_{+}<+\infty$, where
$[t]_{+}=\max\{t,0\}$;\\
(ii) there exists $\varphi^{\ast}\in[0;+\infty)$ such that
$\lim_{k\rightarrow+\infty}\varphi^{k}=\varphi^{\ast}$.

\end{lem}

\section{An inertial primal-dual splitting algorithm
 }
\subsection{Derivation of the algorithm}
In the paper [5], Nesterov proposed a modification of the heavy ball
method in order to improve the convergence rate on smooth convex
functions. The idea of Nesterov was to use the extrapolated point
$y^{k}$ for evaluating the gradient. Moreover, in  order to  prove
optimal convergence rates of the scheme, the extrapolation parameter
$\alpha_{k}$ must satisfy some special conditions. The scheme is
given by:
$$
\left\{
\begin{array}{l}
l^{k}=x^{k}+\alpha_{k}(x^{k}-x^{k-1}),\\
x^{k+1}=l^{k}-\bar{\lambda}_{k}\nabla f(l^{k}),
\end{array}
\right.\eqno{(3.1)}
$$
where $\bar{\lambda}_{k}=1/L$, there are several choices to define
an optimal sequence $\alpha_{k}$ [5-8].

Recently, for Problem (1.1), Condat [1] considered a primal-dual
splitting method as follows:

$$
\left\{
\begin{array}{l}
\tilde{y}^{k+1}=prox_{\sigma h^{\ast}}(y^{k}+\sigma Dx^{k}),\\
\tilde{x}^{k+1}=prox_{\tau g}(x^{k}-\tau \nabla f(x^{k})-\tau
D^{\ast}(2\tilde{y}^{k+1}-y^{k})),\\
(x^{k+1}, y^{k+1})=\rho_{k}(\tilde{x}^{k+1},
\tilde{y}^{k+1})+(1-\rho_{k})(x^{k}, y^{k}),
\end{array}
\right.\eqno{(3.2)}
$$

where  $\sigma>0$, $\tau>0$,  $\frac{1}{\tau}-\sigma\|D\|^{2}>0$,
$\forall k\in \mathbb{N}$, sequences $\rho_{k}\in]0,\delta[$, and
$\delta=2-\frac{l}{2}(\frac{1}{\tau}-\sigma\|D\|^{2})^{-1}\in[1,2[$.
\par
The fixed point characterization provided by Condat [1] suggests
solving Problem (1.1 ) via the fixed point iteration scheme (3.2)
for a suitable value of the parameter  $\sigma>0$, $\tau>0$. This
iteration, which is referred to as a primal-dual splitting method
for convex optimization involving Lipschitzian, proximable and
linear composite terms. A very natural idea is to combine with the
primal-dual splitting method and the heavy ball method, so we obtain
the following Algorithm.
\begin{algorithm}[H]
\caption{An inertial primal-dual splitting algorithm(IPDS).}
\begin{algorithmic}\label{1}
\STATE Initialization: Choose $x^{0}, x^{1}\in \mathcal{X}$, $
y^{0}, y^{1}\in \mathcal{Y}$, relaxation parameters
$(\rho_{k})_{k\in \mathbb{N}}$,\\ ~~~~~~~~~~~~~~~~~~~extrapolation
parameter $\alpha_{k}$ and proximal
 parameters $\sigma>0$, $\tau>0$.\\
Iterations ($k\geq0$): Update $x^{k}$, $y^{k}$ as follows
$$
\left\{
\begin{array}{l}
\xi^{k}=x^{k}+\alpha_{k}(x^{k}-x^{k-1}),\\
\eta^{k}=y^{k}+\alpha_{k}(y^{k}-y^{k-1}),\\
\tilde{y}^{k+1}=prox_{\sigma h^{\ast}}(\eta^{k}+\sigma D\xi^{k}),\\
\tilde{x}^{k+1}=prox_{\tau g}(\xi^{k}-\tau \nabla f(\xi^{k})-\tau
D^{\ast}(2\tilde{y}^{k+1}-\eta^{k})),\\
(x^{k+1}, y^{k+1})=\rho_{k}(\tilde{x}^{k+1},
\tilde{y}^{k+1})+(1-\rho_{k})(x^{k}, y^{k}).
\end{array}
\right.
$$
~~~~~~~~~~~~~~~~~~~End for
\end{algorithmic}
\end{algorithm}
Assume that $\nabla f$ is cocoercive with respect to $L^{-1}(
cf$.(2.4)). Then for Algorithm 1, we given the following Theorem.
\begin{thm}
Let $\sigma>0$, $\tau>0$ , $(\alpha_{k})_{k\in \mathbb{N}}$ and the
sequences $(\rho_{k})_{k\in \mathbb{N}}$, be the parameters of
Algorithms 1. Let $L$ be a linear, bounded, selfadjoint and positive
definite map defined by (2.4) and that the following
hold:\\
(i)  $\frac{1}{\tau}-\sigma\|D\|^{2}>\frac{\|L\|}{2}$,\\
(ii)   $(\alpha_{k})_{k\in \mathbb{N}}$ is nondecreasing with
$\alpha_{1} = 0$ and $0\leq\alpha_{k}\leq\alpha < 1$ for every
$k\geq1$ and $\rho, \theta, \hat{\delta}>0$ are such that
$\hat{\delta}>\frac{\alpha^{2}(1+\alpha)+\alpha\theta}{1-\alpha^{2}}$
and
$0<\rho\leq\rho_{k}<\frac{\hat{\delta}-\alpha[\alpha(1+\alpha)+\alpha\hat{\delta}+\theta]}{\hat{\delta}[1+\alpha(1+\alpha)+\alpha\hat{\delta}+\theta]}$
$\forall k\geq1$.\\
Let the sequences $(x^{k},y^{k})$ be generated by Algorithms 1. Then
the sequence $\{x_{k}\}$ converges to a solution of Problem (1.1).
\end{thm}
We consider the case where $D$ is injective(in particular, it is
implicit that dim$(\mathcal{X})\leq$ dim$(\mathcal{Y}))$. In the
latter case, we denote by $\mathcal{R}$ = Im$(D)$ the image of $D$
and by $D^{-1}$ the inverse of $D$ on
$\mathcal{R}\rightarrow\mathcal{X}$. We emphasize the fact that the
inclusion $\mathcal{R}\subset\mathcal{Y}$ might be strict.  We make
the following assumption:

\begin{ass}
The following facts holds true:\\
 (1)$D$ is injective;\\
 (2)$\nabla(f\circ D)^{-1}$  is cocoercive with respect to $\bar{L}^{-1}(
cf$.(2.4)) on $\mathcal {R}$.
\end{ass}

For proximal parameters $\mu>0$,
 $\tau>0$, we consider the following algorithm which we shall refer to
 as Inertial ADMM$^{+}$( IADMM$^{+}$).
\begin{algorithm}[H]
\caption{Inertial ADMM$^{+}$( IADMM$^{+}$).}
\begin{algorithmic}\label{1}
\STATE Iterations ($k\geq0$): Update $x^{k}$, $y^{k}$ as follows
$$
\left\{
\begin{array}{l}
\xi^{k}=x^{k}+\alpha_{k}(x^{k}-x^{k-1}),\\
\eta^{k}=y^{k}+\alpha_{k}(y^{k}-y^{k-1}),\\
z^{k+1}=\arg\min_{w\in\mathcal {Y}}[h(w)+\frac{\|w-(D\xi^{k}+\mu\eta^{k})\|^{2}}{2\mu}],~~~~~~~~~~~~~~~~~~~~~~~~~~~~~~~~~~~~~~~~~(3.1a)\\
y^{k+1}=\eta^{k}+\mu^{-1}(D\xi^{k}-z^{k+1}),~~~~~~~~~~~~~~~~~~~~~~~~~~~~~~~~~~~~~~~~~~~~~~~~~~~~~~~~~(3.2b)\\
u^{k+1}=(1-\tau\mu^{-1})D\xi^{k}+\tau\mu^{-1}z^{k+1},~~~~~~~~~~~~~~~~~~~~~~~~~~~~~~~~~~~~~~~~~~~~~~~~~~~(3.3c)\\
x^{k+1}=\arg\min_{w\in\mathcal {X}}[g(w)+\langle\nabla
f(\xi^{k}),w\rangle+\frac{\|Dw-u^{k+1}-\tau
y^{k+1}\|^{2}}{2\tau}].~~~~~~~~~~~~~~~~~~~(3.4d)
\end{array}
\right.
$$
End for
\end{algorithmic}
\end{algorithm}
\begin{thm}
Assume that the minimization Problem (1.1) is consistent, $\mu>0$,
and $\tau>0$.  Let Assumption 2.1 and Assumption 3.1 hold true and
 $\bar{L}$ be a linear, bounded, selfadjoint and positive definite map
defined by (2.4) and
$\frac{1}{\tau}-\frac{1}{\mu}>\frac{\|\bar{L}\|}{2}$. Suppose that
 $(\alpha_{k})_{k\in \mathbb{N}}$ is nondecreasing with
$\alpha_{1} = 0$ and $0\leq\alpha_{k}\leq\alpha < 1$ for every
$k\geq1$ and $\rho, \theta, \hat{\delta}>0$ are such that
$\hat{\delta}>\frac{\alpha^{2}(1+\alpha)+\alpha\theta}{1-\alpha^{2}}$
and
$0<\rho\leq\rho_{k}<\frac{\hat{\delta}-\alpha[\alpha(1+\alpha)+\alpha\hat{\delta}+\theta]}{\hat{\delta}[1+\alpha(1+\alpha)+\alpha\hat{\delta}+\theta]}$
$\forall k\geq1$. Let the sequences $(x^{k},y^{k})$ be generated by
Algorithms 2. Then the sequence $\{x^{k}\}$ converges to a solution
of Problem (1.1).
\end{thm}

\subsection{Proofs of convergence}
From the proof of Theorem 3.1  [1], we know that (3.2) has the
structure of a forward-backward iteration, when expressed in terms
of nonexpansive operators on $\mathcal {Z} := \mathcal
{X}\times\mathcal {Y}$, equipped with a particular inner product.
\par
Let the inner product $\langle\cdot,\cdot\rangle_{I}$ in $\mathcal
{Z}$ be defined as
$$\langle z,z'\rangle:=\langle x,x'\rangle+\langle y,y'\rangle,~~~~~\forall z=(x,y),~ z'=(x',y')\in\mathcal {Z}.$$
By endowing $\mathcal {Z}$ with this inner product, we obtain the
Euclidean space denoted by $\mathcal {Z}_{I}$ . Let us define the
bounded linear operator on $\mathcal {Z}$,

$$
P: \left(
  \begin{array}{ccccccc}
    x  \\
    y  \\

  \end{array}
\right)\mapsto\left(
  \begin{array}{ccccccc}
    \frac{1}{\tau} & D^{\ast}\\
    D & \frac{1}{\sigma}\\

  \end{array}
\right)\left(
  \begin{array}{ccccccc}
    x  \\
    y  \\

  \end{array}
\right).\eqno{(3.4)}
$$
From the condition (i), we can easily check that $P$ is positive
definite. Hence, we can define another inner product
$\langle\cdot,\cdot\rangle_{P}$ and norm
$\|\cdot\|_{P}=\langle\cdot,\cdot\rangle_{P}^{\frac{1}{2}}$ in
$\mathcal {Z}$ as
$$\langle z,z'\rangle_{P}=\langle z,Pz'\rangle_{I}.\eqno{(3.5)}$$
We denote by $\mathcal {Z}_{P}$ the corresponding Euclidean space.
\begin{lem}
( [1]). Let the conditions (i)-(iii) in Theorem 3.1[1] be ture . For
every $n\in \mathbb{N}$, the following inclusion is satisfied by
$\tilde{z}^{k+1} := (\tilde{x}^{k+1}, \tilde{y}^{k+1})$ computed by
(3.2):
$$\tilde{z}^{k+1}:=(I+P^{-1}\circ A)^{-1}\circ(I-P^{-1}\circ B)(z^{k}),\eqno{(3.6)}$$
where $$ A:= \left(
  \begin{array}{ccccccc}
    \partial g & D^{\ast} \\
    -D & \partial h^{\ast} \\

  \end{array}
\right),
 B:= \left(
  \begin{array}{ccccccc}
    \nabla f \\
   0 \\

  \end{array}
\right).
$$
Set $M_{1}=P^{-1}\circ A$, $M_{2}=P^{-1}\circ B$,
$T_{1}=(I+M_{1})^{-1}$, $T_{2}=(I-M_{2})^{-1}$, and $T=T_{1}\circ
T_{2}$. Then $T_{1}\in\mathcal {A}(\mathcal {Z}_{P},\frac{1}{2})$
and $T_{2}\in\mathcal {A}(\mathcal {Z}_{P},\frac{1}{2\kappa})$,
$\kappa:=(\frac{1}{\tau}-\sigma\|D\|^{2})/\beta$. Then $T\in\mathcal
{A}(\mathcal {Z}_{P},\frac{1}{\delta})$ and
$\delta=2-\frac{1}{2\kappa}$.

\end{lem}
\begin{lem}
( [9]).Let $\tilde{M}$ be a nonempty closed and affine subset of a
Hilbert space $\mathcal {\bar{H}}$ and $T : \tilde{M} \rightarrow
\tilde{M}$ a nonexpansive operator such that $Fix(T)\neq\emptyset$.
Considering the following iterative scheme:
$$x^{k+1}=x^{k}+\alpha_{k}(x^{k}-x^{k-1})+\rho_{k}[T(x^{k}+\alpha_{k}(x^{k}-x^{k-1}))-x^{k}-\alpha_{k}(x^{k}-x^{k-1})],\eqno{(3.7)}$$
where $x^{0}$; $x^{1}$ are arbitrarily chosen in $\tilde{M}$,
 $(\alpha_{k})_{k\in \mathbb{N}}$ is nondecreasing with
$\alpha_{1} = 0$ and $0\leq\alpha_{k}\leq\alpha < 1$ for every
$n\geq1$ and $\rho, \theta, \hat{\delta}>0$ are such that
$\hat{\delta}>\frac{\alpha^{2}(1+\alpha)+\alpha\theta}{1-\alpha^{2}}$
and
$0<\rho\leq\rho_{k}<\frac{\hat{\delta}-\alpha[\alpha(1+\alpha)+\alpha\hat{\delta}+\theta]}{\hat{\delta}[1+\alpha(1+\alpha)+\alpha\hat{\delta}+\theta]}$
$\forall k\geq1$.\\
Then the following statements are true:\\
(i) $\sum_{k\in\mathbb{N}}\|x^{k+1}-x^{k}\|^{2}<+\infty$;\\
(ii) $(x^{k})_{k\in\mathbb{N}}$ converges weakly to a point in
$Fix(T)$.
\end{lem}

\par
In association with  Lemma 3.1 and Lemma 3.2,  we are ready to prove
Theorem 3.1

\begin{proof}
Set $\nu^{k}:=(\xi^{k},\eta^{k})$, from (3.6) we can know that the
Algorithm 1 can be described as follows:
$$
\left\{
\begin{array}{l}
\nu^{k}=z^{k}+\alpha_{k}(z^{k}-z^{k-1}),\\
\tilde{z}^{k+1}:=(I+P^{-1}\circ A)^{-1}\circ(I-P^{-1}\circ
B)(\nu^{k}).
\end{array}
\right.\eqno{(3.8)}
$$
Considering  the relaxation step, we obtain
$$
\left\{
\begin{array}{l}
\nu^{k}=z^{k}+\alpha_{k}(z^{k}-z^{k-1}),\\
\tilde{z}^{k+1}:=(I+P^{-1}\circ A)^{-1}\circ(I-P^{-1}\circ
B)(\nu^{k}),\\
z^{k+1}:=\rho_{k}(I+P^{-1}\circ A)^{-1}\circ(I-P^{-1}\circ
B)(\nu^{k})+(1-\rho_{k})\nu^{k}.
\end{array}
\right.\eqno{(3.9)}
$$

By Lemma 3.1 we know that $T=T_{1}\circ T_{2}$ is
$\frac{1}{\delta}$-averaged. In particular, it is non-expansive, so
from  conditions (i)-(ii)  and Lemma 3.2 we have that the  iterative
scheme defined by (3.9)
satisfies the following statements:\\
(i) $\sum_{k\in\mathbb{N}}\|z^{k+1}-z^{k}\|^{2}<+\infty$;\\
(ii) $(z^{k})_{k\in\mathbb{N}}$ converges to a point in $Fix(T)$.\\
Then the sequence $\{x^{k}\}$ converges to a solution of Problem
(1.1).

\end{proof}
Proof of Theorem 3.2 for Algorithm 2. Before providing the proof of
Theorem 3.2, let us introduce the following notation and Lemma.
\begin{lem}
 Given a Euclidean space $\mathcal {E}$, consider
the minimization problem $\min_{\lambda\in\mathcal{E}}
\bar{f}(\lambda)+\bar{g}(\lambda)+ h(\lambda)$, where $\bar{g}, h
\in \Gamma_{0}(\mathcal {E})$ and where $\bar{f}$ is convex and
differentiable on $\mathcal {E}$ and $\nabla \bar{f}$ is cocoercive
with respect to $L^{-1}( cf$.(2.4)). Assume that the infimum is
attained and that $0 \in ri(dom h-dom\bar{g})$. Let $\mu>0$,
$\tau>0$ be such that
$\frac{1}{\tau}-\frac{1}{\mu}>\frac{\|\bar{L}\|}{2}$
 ,    $(\alpha_{k})_{k\in \mathbb{N}}$ be nondecreasing with
$\alpha_{1} = 0$ and $0\leq\alpha_{k}\leq\alpha < 1$ for every
$n\geq1$ and $\rho, \theta, \hat{\delta}>0$ are such that
$\hat{\delta}>\frac{\alpha^{2}(1+\alpha)+\alpha\theta}{1-\alpha^{2}}$
and
$0<\rho\leq\rho_{k}<\frac{\hat{\delta}-\alpha[\alpha(1+\alpha)+\alpha\hat{\delta}+\theta]}{\hat{\delta}[1+\alpha(1+\alpha)+\alpha\hat{\delta}+\theta]}$
$\forall k\geq1$. Consider the iterates
$$
\left\{
\begin{array}{l}
\bar{\xi}^{k}=\lambda^{k}+\alpha_{k}(\lambda^{k}-\lambda^{k-1}),\\
\eta^{k}=y^{k}+\alpha_{k}(y^{k}-y^{k-1}),\\
y^{k+1}=prox_{\mu^{-1} h^{\ast}}(\eta^{k}+\mu^{-1}\bar{\xi}^{k}),~~~~~~~~~~~~~~~~~~~~~~~~~~~~~~~~~~~~~~~~~~~~~~~~~~~~~~~~(3.10a)\\
\lambda^{k+1}=prox_{\tau \bar{g}}(\bar{\xi}^{k}-\tau \nabla
\bar{f}(\bar{\xi}^{k})-\tau
(2y^{k+1}-\eta^{k})).~~~~~~~~~~~~~~~~~~~~~~~~~~~~~~~~~~~~~(3.10b)
\end{array}
\right.
$$
Then for any initial value $(\lambda^{0}, y^{0}), (\lambda^{1},
y^{1}) \in \mathcal {E}\times \mathcal {E}$, the sequence
$(\lambda^{k}, y^{k})$ converges to a primal-dual point
$(\tilde{\lambda}, \tilde{y})$, i.e., a solution of the equation
$$\min_{\lambda\in\mathcal{E}} \bar{f}(\lambda)+\bar{g}(\lambda)+
h(\lambda)=-\min_{y\in\mathcal{E}} (\bar{f}+\bar{g})^{\ast}(y)+
h^{\ast}(y).\eqno{(3.11)}$$
\end{lem}
\begin{proof}
It is easy to see that the Lemma 3.3 is a special case of Theorem
3.1. So we can obtain Lemma 3.3 from Theorem 3.1 directly.

\end{proof}
Elaborating on Lemma 3.3, we are now ready to establish the Theorem
3.2.

By setting $\mathcal {E} = \mathcal {R}$ and by assuming that
$\mathcal {E}$ is equipped with the same inner product as $\mathcal
{Y}$, one can notice that the functions $\bar{ f} = f \circ D^{-1}$,
$\bar{ g} = g \circ D^{-1}$ and $h$ satisfy the conditions of Lemma
3.3. Moreover, since $(\bar{f}+\bar{g})^{\ast}= (f + g)^{\ast}\circ
D^{\ast}$, one can also notice that $(\tilde{x}, \tilde{y})$ is a
primal-dual point associated with Eq. (2.1) if and only if
$(D\tilde{x}, \tilde{y})$ is a primal-dual point associated with Eq.
(3.11). With the same idea for the proof of Theorem 1 of  [2], we
can recover the IADMM$^{+}$ from the iterations (3.10).
\subsection{Connections to other algorithms}
We will further establish the connections to other existing methods.

When $\alpha_{k}\equiv0$ , the IADMM$^{+}$ boils down to the
ADMM$^{+}$ whose iterations are given by:
$$
\left\{
\begin{array}{l}
z^{k+1}=argmin_{w\in\mathcal {Y}}[h(w)+\frac{\|w-(Dx^{k}+\mu y^{k})\|^{2}}{2\mu}],\\
y^{k+1}=y^{k}+\mu^{-1}(Dx^{k}-z^{k+1}),\\
u^{k+1}=(1-\tau\mu^{-1})Dx^{k}+\tau\mu^{-1}z^{k+1},\\
x^{k+1}=argmin_{w\in\mathcal {X}}[g(w)+\langle\nabla
f(x^{k}),w\rangle+\frac{\|Dw-u^{k+1}-\tau y^{k+1}\|^{2}}{2\tau}].
\end{array}
\right.
$$
In the special case   $h \equiv 0$ , $D = I$ and
 $\alpha_{k}\equiv0$ it can be easily verified that $y^{k}$ is null
for all $k \geq 1$ and $u^{k} = x^{k}$. Then, the IADMM$^{+}$ boils
down to the standard Forward-Backward algorithm whose iterations are
given by:
\begin{align*}
x^{k+1}&=argmin_{w\in\mathcal
{X}}g(w)+\frac{1}{2\tau}\|w-(x^{k}-\tau \nabla
f(x^{k}))\|^{2}\\
&=prox_{\tau g}(x^{k}-\tau \nabla f(x^{k})).
\end{align*}
One can remark that $\mu$  has disappeared thus it can be set as
large as wanted so the condition on stepsize $\tau$ from Theorem 3.2
boils down to $\tau < 2/l$. Applications of this algorithm with
particular functions appear in well known learning methods such as
ISTA [10].
\section{Preconditioning}
\subsection{Convergence of the Preconditioned algorithm}
In the context of saddle point problems, Pock and Chambolle [12]
proposed a preconditioning of the form

$$ P:= \left(
  \begin{array}{ccccccc}
    \tilde{T}^{-1} & D^{\ast} \\
    D & \Sigma^{-1} \\

  \end{array}
\right)
$$

where $\tilde{T}$ and $\Sigma$ are selfadjoint, positive definite
maps. A condition for the positive definiteness of $P$ follows from
the following Lemma.
\begin{lem}
([12]). Let $A_{1}$, $A_{2}$ be symmetric positive definite maps and
$M$ a bounded operator. If
$\|A_{2}^{-\frac{1}{2}}MA_{1}^{-\frac{1}{2}} \| < 1$, then
$$ A:= \left(
  \begin{array}{ccccccc}
    A_{1} & M^{\ast} \\
    M & A_{2} \\

  \end{array}
\right)
$$

is positive definite.
\end{lem}

 Now, we study preconditioning
techniques for the inertial primal-dual splitting algorithm(IPDS),
then we obtain the following algorithm.
\begin{algorithm}[H]
\caption{An inertial primal-dual splitting algorithm with
preconditioning (IPDSP).}
\begin{algorithmic}\label{1}
\STATE Initialization: Choose $x^{0}, x^{1}\in \mathcal{X}$, $
y^{0}, y^{1}\in \mathcal{Y}$, relaxation parameters
$(\rho_{k})_{k\in \mathbb{N}}$,\\ ~~~~~~~~~~~~~~~~~~~extrapolation
parameter $\alpha_{k}$ and positive definite maps $\tilde{T}$, $\Sigma$.\\
Iterations ($k\geq0$): Update $x^{k}$, $y^{k}$ as follows
$$
\left\{
\begin{array}{l}
\xi^{k}=x^{k}+\alpha_{k}(x^{k}-x^{k-1}),\\
\eta^{k}=y^{k}+\alpha_{k}(y^{k}-y^{k-1}),\\
\tilde{y}^{k+1}=prox_{\Sigma h^{\ast}}(\eta^{k}+\Sigma D\xi^{k}),\\
\tilde{x}^{k+1}=prox_{\tilde{T} g}(\xi^{k}-\tilde{T} \nabla
f(\xi^{k})-\tilde{T}
D^{\ast}(2\tilde{y}^{k+1}-\eta^{k})),\\
(x^{k+1}, y^{k+1})=\rho_{k}(\tilde{x}^{k+1},
\tilde{y}^{k+1})+(1-\rho_{k})(x^{k}, y^{k}).
\end{array}
\right.
$$
~~~~~~~~~~~~~~~~~~~End for
\end{algorithmic}
\end{algorithm}
It turns out that the resulting method converges under appropriate
conditions.
\begin{thm}
In the setting of Theorem 3.1 let furthermore  $\nabla f$  be
co-coercive w.r.t. a bound, linear, symmetric and positive linear
maps $E^{-1}$. If it holds that\\
(i) $\tilde{T}^{-1}-\frac{1}{2}E>0;$\\
(ii)  $\|\tilde{T}^{-1}\|-\|\Sigma\|\|D\|^{2}>\frac{\|E\|}{2}$,\\
(iii)   $(\alpha_{k})_{k\in \mathbb{N}}$ is nondecreasing with
$\alpha_{1} = 0$ and $0\leq\alpha_{k}\leq\alpha < 1$ for every
$n\geq1$ and $\rho, \theta, \hat{\delta}>0$ are such that
$\hat{\delta}>\frac{\alpha^{2}(1+\alpha)+\alpha\theta}{1-\alpha^{2}}$
and
$0<\rho\leq\rho_{k}<\frac{\hat{\delta}-\alpha[\alpha(1+\alpha)+\alpha\hat{\delta}+\theta]}{\hat{\delta}[1+\alpha(1+\alpha)+\alpha\hat{\delta}+\theta]}$
$\forall k\geq1$.\\
Then the sequence $\{x^{k}\}$ converges to a solution of Problem
(1.1).

\end{thm}
\begin{proof}
It is easy to check that from the condition (i)-(ii),  we can obtain
$\|\tilde{T}^{-1}-\frac{1}{2}E)^{-\frac{1}{2}}D\Sigma^{\frac{1}{2}}\|<1$.

 Set
$$ C:= \left(
  \begin{array}{ccccccc}
    E & 0 \\
    0 & 0 \\

  \end{array}
\right)
$$

 Then from Lemma 4.1, we can know that $P-\frac{1}{2}C$ is positive
 definite. Therefore, with the same proof of Theorem 3.1, we can
 obtain Theorem 4.1.

\end{proof}
For selfadjoint, positive definite maps $\tilde{T}$, $\Psi$, we
consider the following algorithm which we shall refer to as
Preconditioning ADMM$^{+}$(PADMM$^{+}$).

\begin{algorithm}[H]
\caption{Preconditioning ADMM$^{+}$(PADMM$^{+}$).}
\begin{algorithmic}\label{1}
\STATE Iterations ($k\geq0$): Update $x^{k}$, $y^{k}$ as follows
$$
\left\{
\begin{array}{l}
\xi^{k}=x^{k}+\alpha_{k}(x^{k}-x^{k-1}),\\
\eta^{k}=y^{k}+\alpha_{k}(y^{k}-y^{k-1}),\\
z^{k+1}=\arg\min_{w\in\mathcal {Y}}[h(w)+\frac{\|w-(D\xi^{k}+\Psi\eta^{k})\|_{\Psi^{-1}}^{2}}{2}],~~~~~~~~~~~~~~~~~~~~~~~~~~~~~~~~~~~~~~~(4.1a)\\
y^{k+1}=\eta^{k}+\Psi^{-1}(D\xi^{k}-z^{k+1}),~~~~~~~~~~~~~~~~~~~~~~~~~~~~~~~~~~~~~~~~~~~~~~~~~~~~~~~~~~(4.1b)\\
u^{k+1}=(I-\tilde{T}\Psi^{-1})D\xi^{k}+\tilde{T}\Psi^{-1}z^{k+1},~~~~~~~~~~~~~~~~~~~~~~~~~~~~~~~~~~~~~~~~~~~~~~~~~~(4.1c)\\
x^{k+1}=\arg\min_{w\in\mathcal {X}}[g(w)+\langle\nabla
f(\xi^{k}),w\rangle+\frac{\|Dw-u^{k+1}-\tilde{T}
y^{k+1}\|_{\tilde{T}^{-1}}^{2}}{2}].~~~~~~~~~~~~~~~~~(4.1d)
\end{array}
\right.
$$
End for
\end{algorithmic}
\end{algorithm}
\begin{thm}
In the setting of Theorem 3.2 let furthermore  $\nabla \bar{f}$  be
co-coercive w.r.t. a bound, linear, symmetric and positive linear
maps $\bar{E}^{-1}$. If it holds that\\
(i) $\tilde{T}^{-1}-\frac{1}{2}\bar{E}>0;$\\
(ii)  $\|\tilde{T}^{-1}\|-\|\Psi^{-1}\|>\frac{\|\bar{E}\|}{2}$,\\
(iii)   $(\alpha_{k})_{k\in \mathbb{N}}$ is nondecreasing with
$\alpha_{1} = 0$ and $0\leq\alpha_{k}\leq\alpha < 1$ for every
$n\geq1$ and $\rho, \theta, \hat{\delta}>0$ are such that
$\hat{\delta}>\frac{\alpha^{2}(1+\alpha)+\alpha\theta}{1-\alpha^{2}}$
and
$0<\rho\leq\rho_{k}<\frac{\hat{\delta}-\alpha[\alpha(1+\alpha)+\alpha\hat{\delta}+\theta]}{\hat{\delta}[1+\alpha(1+\alpha)+\alpha\hat{\delta}+\theta]}$
$\forall k\geq1$.\\
Then the sequence $\{x^{k}\}$ converges to a solution of Problem
(1.1).

\end{thm}
\begin{proof}
It is easy to check that from the condition (i)-(ii),  we can obtain
$\|(\tilde{T}^{-1}-\frac{1}{2}\bar{E})^{-\frac{1}{2}}\Psi^{-\frac{1}{2}}\|<1$.

 Set
$$ \bar{P}:= \left(
  \begin{array}{ccccccc}
    \tilde{T}^{-1} & I \\
    I & \Psi \\

  \end{array}
\right), \bar{C}:= \left(
  \begin{array}{ccccccc}
    \bar{E} & 0 \\
    0 & 0 \\

  \end{array}
\right).
$$

 Then from Lemma 4.1, we can know that $\bar{P}-\frac{1}{2}\bar{C}$ is positive
 definite. Therefore, with the same proof of Theorem 3.2, we can
 obtain Theorem 4.2.

\end{proof}
\subsection{Diagonal Preconditioning}
In this section, we show how we can choose pointwise step sizes for
both the primal and the dual variables that will ensure the
convergence of the algorithm. The next result is an adaption of the
preconditioner proposed in [11].

\begin{lem}
Assume that $\nabla f$ is co-coercive with respect to diagonal
matrices $E^{-1}$, where $E=diag(e_{1},\cdots,e_{n})$. Fix $\gamma
\in (0, 2)$, $r > 0, s \in [0, 2]$ and let $\tilde{T} =
diag(\tau_{1},\cdots , \tau_{n})$ and $\Psi= diag(\varphi_{1},
\cdots , \varphi_{m})$ with

$$\tau_{j}=\frac{1}{\frac{e_{j}}{\gamma}+\sum_{i=1}^{m}|D_{i,j}|^{2-s}},\varphi_{i}=\frac{1}{r}\sum_{j=1}^{n}|D_{i,j}|^{s},\eqno{(4.2)}$$
then it holds that\\
$$\tilde{T}^{-1}-\frac{1}{2}E>0,\Psi>0,\eqno{(4.3)}$$
$$\|\tilde{T}^{-1}\|-\|\Psi^{-1}\|>\frac{\|\bar{E}\|}{2}.\eqno{(4.4)}$$

\end{lem}
\begin{proof}
The first two conditions follow from the fact that for diagonal
matrices, the (4.3) can be written pointwise. By the definition of
$\tau_{j}$ , and $\varphi_{i}$ it follows that for any $s \in [0,
2]$ and using the convention that $0^{0} = 0$,
$$\frac{1}{\tau_{j}}-\frac{e_{j}}{2}>\frac{1}{\tau_{j}}-\frac{e_{j}}{\gamma}=r\sum_{i=1}^{m}|D_{i,j}|^{2-s}\geq0,$$
and
$$\varphi_{i}=\frac{1}{r}\sum_{j=1}^{n}|D_{i,j}|^{s}\geq0.$$
 We will prove(4.4). It is easy to see
 the proof of (4.4) is equivalent to the proof of
$$\|\Psi^{-\frac{1}{2}}D(\tilde{T}^{-1}-\frac{1}{2}E)^{-\frac{1}{2}}\|<1.\eqno{(4.5)}$$
So we first show (4.5), For any  $s \in [0, 2]$,
\begin{align*}
\|\Psi^{-\frac{1}{2}}D(\tilde{T}^{-1}-\frac{1}{2}E)^{-\frac{1}{2}}x\|^{2}&=\sum_{i=1}^{m}(\sum_{j=1}^{n}\frac{1}{\sqrt{\varphi_{i}}}D_{i,j}\frac{1}{\sqrt{\frac{1}{\tau_{j}}-\frac{e_{j}}{2}}}x_{j})^{2}\\
&=\sum_{i=1}^{m}\frac{1}{\varphi_{i}}(\sum_{j=1}^{n}D_{i,j}\frac{1}{\sqrt{\frac{1}{\tau_{j}}-\frac{e_{j}}{2}}}x_{j})^{2}\\
&\leq\sum_{i=1}^{m}\frac{1}{\varphi_{i}}(\sum_{j=1}^{n}|D_{i,j}|^{\frac{s}{2}}|D_{i,j}|^{1-\frac{s}{2}}\frac{1}{\sqrt{\frac{1}{\tau_{j}}-\frac{e_{j}}{2}}}x_{j})^{2}\\
&\leq\sum_{i=1}^{m}\frac{1}{\varphi_{i}}(\sum_{j=1}^{n}|D_{i,j}|^{s})(\sum_{j=1}^{n}|D_{i,j}|^{2-s}\frac{1}{\sqrt{\frac{1}{\tau_{j}}-\frac{e_{j}}{2}}}x_{j}^{2}).\tag{4.6}
\end{align*}
By definition of $\tau_{j}$ and $\varphi_{i}$ , and introducing
$r>0$, the above estimate can be simplified to
\begin{align*}
&\sum_{i=1}^{m}\frac{\frac{1}{r}}{\varphi_{i}}(\sum_{j=1}^{n}|D_{i,j}|^{s})(\sum_{j=1}^{n}|D_{i,j}|^{2-s}\frac{r}{\sqrt{\frac{1}{\tau_{j}}-\frac{e_{j}}{2}}}x_{j}^{2})\\
&=\sum_{i=1}^{m}\sum_{j=1}^{n}|D_{i,j}|^{2-s}\frac{r}{\sqrt{\frac{1}{\tau_{j}}-\frac{e_{j}}{2}}}x_{j}^{2}\\
&\leq\sum_{j=1}^{n}(\sum_{i=1}^{m}|D_{i,j}|^{2-s})\frac{r}{\sqrt{\frac{1}{\tau_{j}}-\frac{e_{j}}{2}}}x_{j}^{2}=\|x\|^{2}.\tag{4.7}
\end{align*}
Using the above estimate in the definition of the operator norm, we
obtain the desired result
\begin{align*}
&\|\Psi^{-\frac{1}{2}}D(\tilde{T}^{-1}-\frac{1}{2}E)^{-\frac{1}{2}}\|^{2}\\
&=\sup_{x\neq0}\frac{\|\Psi^{-\frac{1}{2}}D(\tilde{T}^{-1}-\frac{1}{2}E)^{-\frac{1}{2}}x\|^{2}}{\|x\|^{2}}\leq1.\tag{4.8}
\end{align*}

\end{proof}

\begin{rmk}
In particular, for $D=I_{\mathcal {Y}}$,  we obtain that
$$\tau_{j}=\frac{1}{\frac{e_{j}}{\gamma}+n},\varphi_{i}=\frac{1}{r}n,\eqno{(4.9)}$$
then it also holds (4.3) and (4.4).
\end{rmk}

For $D=I_{\mathcal {Y}}$, from Theorem 4.2, we know that the
PADMM$^{+}$ iterates are generated by the action of a nonexpansive
operator. Then by Lemma 2.3 we know that a stochastic coordinate
descent version of the PADMM$^{+}$ converges towards a primal-dual
point. This result will be exploited in two directions: first, we
describe a stochastic minibatch algorithm, where a large dataset is
randomly split into smaller chunks. Second, we develop an
asynchronous version of the PADMM$^{+}$ in the context where it is
distributed on a graph.
\section{Coordinate descent}
\subsection{Randomized krasnosel'skii-mann iterations}
 Consider the space
$\mathcal{Z}=\mathcal{Z}_{1}\times\cdots\times\mathcal{Z}_{J}$ for
some $J\in\mathbb{N}^{\ast}$ where for any $j$, $\mathcal{Z}_{j}$ is
a Euclidean space. For $\mathcal{Z}$ equipped with the scalar
product $\langle x,y\rangle=\sum_{j=1}^{J}\langle
x_{j},y_{j}\rangle_{\mathcal{Z}_{j}}$ where $\langle
\cdot,\cdot\rangle_{\mathcal{Z}_{j}}$ is the scalar product in
$\mathcal{Z}_{j}$. For $j\in \{1,\cdots,J\}$ , let $T_{j}:
\mathcal{Z}\rightarrow\mathcal{Z}_{j}$ be the components of the
output of operator $T : \mathcal{Z}\rightarrow\mathcal{Z}$
corresponding to $\mathcal{Z}_{j}$ , so, we have
$Tx=(T_{1}x,\cdots,T_{J}x)$. Let $2^{\mathcal{J}}$ be the power set
of $\mathcal{J}=\{1,\cdots,J\}$. For any $\vartheta\in
2^{\mathcal{J}}$, we donate the operator $\hat{T}^{\vartheta}:
\mathcal{Z}\rightarrow\mathcal{Z}$ by
$\hat{T}^{\vartheta}_{j}x=T_{j}x$ for $j\in\vartheta$ and
$\hat{T}^{\vartheta}_{j}x=x_{j}$ for otherwise. On some probability
space $(\Omega, \mathcal{F}, \mathbb{P})$, we introduce a random
i.i.d. sequence $(\zeta^{k})_{k\in\mathbb{N}^{\ast}}$ such that
$\zeta^{k}:\Omega\rightarrow 2^{\mathcal{J}}$ i.e.
$\zeta^{k}(\omega)$ is a subset of $\mathcal{J}$. Assume that the
following holds:
$$\forall j\in \mathcal{J},\exists \vartheta\in 2^{\mathcal{J}}, j\in\vartheta~~~ and ~~~\mathbb{P}(\zeta_{1}=\vartheta)>0.\eqno{(5.1)}$$

\begin{lem}
(Theorem 3 of [2]). Let $T: \mathcal{Z}\rightarrow\mathcal{Z}$ be
$\tilde{a}$-averaged and Fix(T)$\neq\emptyset$. Let
$(\zeta^{k})_{k\in\mathbb{N}^{\ast}}$ be a random i.i.d. sequence on
$2^{\mathcal{J}}$ such that Condition (5.1) holds. If for all $k$,
sequence $(\rho_{k})_{k\in\mathbb{N}}$ satisfies
$$0<\liminf_{k\rightarrow\infty}\rho_{k}\leq\limsup_{k\rightarrow\infty}\rho_{k}<\frac{1}{\tilde{a}}.$$
Then, almost surely, the iterated sequence

$$x^{k+1}=x^{k}+\rho_{k}(\hat{T}^{(\zeta^{k+1})}x^{k}-x^{k})\eqno{(5.2)}$$
converges to some point in Fix($T$).
\end{lem}
In particular, if  $T$ is nonexpansive,  and for all $k$, sequence
$(\rho_{k})_{k\in\mathbb{N}}$ satisfies
$$0<\liminf_{k}\rho_{k}\leq\limsup_{k}\rho_{k}<1.$$
We can know the iterated sequence (5.2) converges to some point in
Fix($T$ ). Then we obtain the following theorem.
\begin{thm}
 Let $T: \mathcal{Z}\rightarrow\mathcal{Z}$ be
nonexpansive and Fix(T)$\neq\emptyset$. Let
$(\zeta^{k})_{k\in\mathbb{N}^{\ast}}$ be a random i.i.d. sequence on
$2^{\mathcal{J}}$ such that Condition (5.1) holds. We consider the
following iterative scheme:
$$x^{k+1}=x^{k}+\alpha_{k}(x^{k}-x^{k-1})+\rho_{k}[\hat{T}^{(\zeta^{k+1})}(x^{k}+\alpha_{k}(x^{k}-x^{k-1}))-x^{k}-\alpha_{k}(x^{k}-x^{k-1})],\eqno{(5.3)}$$
where $x^{0}$; $x^{1}$ are arbitrarily chosen in $\mathcal{Z}$,
$(\alpha_{k})_{k\in\mathbb{N}}$ is nondecreasing with $\alpha_{1} =
0$ and $0\leq \alpha_{k}\leq\alpha < 1$ for every $k \geq 1$ and
$\rho; \theta; \hat{\delta}
> 0$ are such that\\
(i) $\hat{\delta}>\frac{\alpha^{2}(1+\alpha)+\alpha\theta}{1-\alpha^{2}}$;\\
(ii)
$0<\rho\leq\rho_{k}<\frac{\hat{\delta}-\alpha[\alpha(1+\alpha)+\alpha\hat{\delta}+\theta]}{\hat{\delta}[1+\alpha(1+\alpha)+\alpha\hat{\delta}+\theta]}$
$\forall k\geq1$.\\
 Then, almost surely, the iterated sequence $\{x^{k}\}$
converges to some point in Fix($T$).
\end{thm}
\begin{proof}
Let us start with the remark that, due to the choice of
$\hat{\delta}$, $\rho_{k}\in(0,1)$ for every $k \geq1$. Set
$p_{\vartheta}=\mathbb{P}(\zeta_{1}=\vartheta)$ for any
$\vartheta\in 2^{\mathcal{J}}$. Denote by $\|x\|^{2} = \langle x,
x\rangle$ the squared norm in $\mathcal{Z}$. Define a new inner
product $x\bullet y=\sum_{j=1}^{J}q_{j}\langle x_{j},
y_{j}\rangle_{j}$ on $\mathcal{Z}$ where
$q_{j}^{-1}=\sum_{\vartheta\in2^{\mathcal{J}}}p_{\vartheta}\mathbf{1}_{\{j\in\vartheta\}}$
and let $\||x\||^{2}= x\bullet x$ be its associated squared norm.
Denote by $w^{k}=x^{k}+\alpha_{k}(x^{k}-x^{k-1})$. Consider any
$\tilde{x}\in Fix(T)$. It follows from Lemma 2.3 and conditionally
to the sigma-field $\mathcal {F}^{k} = \sigma(\zeta_{1}, . . . ,
\zeta^{k})$ we have
\begin{align*}
\mathbb{E}[\||x^{k+1}-\tilde{x}\||^{2}|\mathcal
{F}^{k}]&=\sum_{\vartheta\in2^{\mathcal{J}}}p_{\vartheta}\||w^{k}+\rho_{k}[\hat{T}^{(\zeta^{k+1})}w^{k}-w^{k}]-\tilde{x}\||^{2}\\
&=(1-\rho_{k})\||w^{k}-\tilde{x}\||^{2}+\rho_{k}\sum_{\vartheta\in2^{\mathcal{J}}}p_{\vartheta}\||\hat{T}^{(\zeta^{k+1})}w^{k}-\tilde{x}\||^{2}\\
&-\rho_{k}(1-\rho_{k})\sum_{\vartheta\in2^{\mathcal{J}}}p_{\vartheta}\||\hat{T}^{(\zeta^{k+1})}w^{k}-w^{k}\||^{2}\\
&=(1-\rho_{k})\||w^{k}-\tilde{x}\||^{2}+\rho_{k}[\sum_{\vartheta\in2^{\mathcal{J}}}p_{\vartheta}\sum_{j\in\vartheta}q_{j}\|T_{j}w^{k}-\tilde{x}_{j}\|^{2}\\
&+\sum_{\vartheta\in2^{\mathcal{J}}}p_{\vartheta}\sum_{j\neq\vartheta}q_{j}\|w^{k}_{j}-\tilde{x}_{j}\|^{2}]\\
&-\rho_{k}(1-\rho_{k})\sum_{\vartheta\in2^{\mathcal{J}}}p_{\vartheta}\||\hat{T}^{(\zeta^{k+1})}w^{k}-w^{k}\||^{2}\\
&=(1-\rho_{k})\||w^{k}-\tilde{x}\||^{2}+\rho_{k}\||w^{k}-\tilde{x}\||^{2}\\
&+\rho_{k}\sum_{\vartheta\in2^{\mathcal{J}}}p_{\vartheta}\sum_{j\in\vartheta}q_{j}[\|T_{j}w^{k}-\tilde{x}_{j}\|^{2}-\|w^{k}_{j}-\tilde{x}_{j}\|^{2}]\\
&-\rho_{k}(1-\rho_{k})\sum_{\vartheta\in2^{\mathcal{J}}}p_{\vartheta}\||\hat{T}^{(\zeta^{k+1})}w^{k}-w^{k}\||^{2}\\
&=\||w^{k}-\tilde{x}\||^{2}+\rho_{k}\sum_{j=1}^{J}[\|T_{j}w^{k}-\tilde{x}_{j}\|^{2}-\|w^{k}_{j}-\tilde{x}_{j}\|^{2}]\\
&-\rho_{k}(1-\rho_{k})\sum_{\vartheta\in2^{\mathcal{J}}}p_{\vartheta}\||\hat{T}^{(\zeta^{k+1})}w^{k}-w^{k}\||^{2}\\
&=\||w^{k}-\tilde{x}\||^{2}+\rho_{k}[\|Tw^{k}-\tilde{x}\|^{2}-\|w^{k}-\tilde{x}\|^{2}]\\
&-\rho_{k}(1-\rho_{k})\sum_{\vartheta\in2^{\mathcal{J}}}p_{\vartheta}\||\hat{T}^{(\zeta^{k+1})}w^{k}-w^{k}\||^{2}.
\end{align*}
By the the nonexpansiveness of $T$, we have
\begin{align*}
\mathbb{E}[\||x^{k+1}-\tilde{x}\||^{2}|\mathcal
{F}^{k}]&\leq\||w^{k}-\tilde{x}\||^{2}\\
&-\rho_{k}(1-\rho_{k})\sum_{\vartheta\in2^{\mathcal{J}}}p_{\vartheta}\||\hat{T}^{(\zeta^{k+1})}w^{k}-w^{k}\||^{2}.\tag{5.4}\\
\end{align*}
Applying again Lemma 2.3 we have
\begin{align*}
\||w^{k}-\tilde{x}\||^{2}&=\||(1+\alpha_{k})(x^{k}-\tilde{x})-\alpha_{k}(x^{k-1}-\tilde{x})\||^{2}\\
&=(1+\alpha_{k})\||x^{k}-\tilde{x}\||^{2}-\alpha_{k}\||x^{k-1}-\tilde{x}\||^{2}+\alpha_{k}(1+\alpha_{k})\||x^{k}-x^{k-1}\||^{2},\\
\end{align*}
hence by (5.4) we obtain
\begin{align*}
&\mathbb{E}[\||x^{k+1}-\tilde{x}\||^{2}|\mathcal {F}^{k}]-(1+\alpha_{k})\||x^{k}-\tilde{x}\||^{2}+\alpha_{k}\||x^{k-1}-\tilde{x}\||^{2}\\
&\leq-\rho_{k}(1-\rho_{k})\sum_{\vartheta\in2^{\mathcal{J}}}p_{\vartheta}\||\hat{T}^{(\zeta^{k+1})}w^{k}-w^{k}\||^{2}+\alpha_{k}(1+\alpha_{k})\||x^{k}-x^{k-1}\||^{2}.\tag{5.5}\\
\end{align*}
Further, we have
\begin{align*}
\sum_{\vartheta\in2^{\mathcal{J}}}p_{\vartheta}\||\hat{T}^{(\zeta^{k+1})}w^{k}-w^{k}\||^{2}&=\||\frac{1}{\rho_{k}}(x^{k+1}-x^{k})+\frac{\alpha_{k}}{\rho_{k}}(x^{k-1}-x^{k})\||^{2}\\
&\geq\frac{1}{\rho_{k}^{2}}\||x^{k+1}-x^{k}\||^{2}+\frac{\alpha_{k}^{2}}{\rho_{k}^{2}}\||x^{k-1}-x^{k}\||^{2}\\
&+\frac{\alpha_{k}}{\rho_{k}^{2}}(-\lambda_{k}\||x^{k+1}-x^{k}\||^{2}-\frac{1}{\lambda_{k}}\||x^{k-1}-x^{k}\||^{2}),\tag{5.6}
\end{align*}
where we denote
$\lambda_{k}=\frac{1}{\alpha_{k}+\hat{\delta}\rho_{k}}$.

 We derive from (5.5) and (5.6) the inequality (notice that $\rho_{k} \in(0;
1)$)
\begin{align*}
&\mathbb{E}[\||x^{k+1}-\tilde{x}\||^{2}|\mathcal {F}^{k}]-(1+\alpha_{k})\||x^{k}-\tilde{x}\||^{2}+\alpha_{k}\||x^{k-1}-\tilde{x}\||^{2}\\
&\leq\frac{(1-\rho_{k})(\alpha_{k}\lambda_{k}-1)}{\rho_{k}}\||x^{k+1}-x^{k}\||^{2}+\gamma_{k}\||x^{k}-x^{k-1}\||^{2},\tag{5.7}\\
\end{align*}
where
$$\gamma_{k}:=\alpha_{k}(1+\alpha_{k})+\alpha_{k}(1-\rho_{k})\frac{1-\lambda_{k}\alpha_{k}}{\lambda_{k}\rho_{k}}>0.\eqno{(5.8)}$$
Taking again into account the choice of $\lambda_{k}$ we have
$$\hat{\delta}=\frac{1-\lambda_{k}\alpha_{k}}{\lambda_{k}\rho_{k}},$$
and by (5.8) it follows
$$\gamma_{k}:=\alpha_{k}(1+\alpha_{k})+\alpha_{k}(1-\rho_{k})\hat{\delta}\leq \alpha(1+\alpha)+\alpha\hat{\delta}, \forall k\geq1.\eqno{(5.9)}$$

In the following we use some techniques from [14] adapted to our
setting. We define the sequences
$\varphi^{k}:=\||x^{k}-\tilde{x}\||^{2}$ for all $k\in\mathbb{N}$
and
$\varpi^{k}:=\varphi^{k}-\alpha_{k}\varphi^{k-1}+\gamma_{k}\||x^{k}-x^{k-1}\||^{2}$,
for all $k\geq1$. Using the monotonicity of $(\alpha_{k})_{k\geq1}$
and the fact that $\varphi^{k}>0$ for all $k\in\mathbb{N}$, we get
$$\varpi^{k+1}-\varpi^{k}\leq\varphi^{k+1}-(1+\alpha_{k})\varphi^{k}+\alpha_{k}\varphi^{k-1}+\gamma_{k+1}\||x^{k+1}-x^{k}\||^{2}-\gamma_{k}\||x^{k}-x^{k-1}\||^{2},$$
which gives by (5.7)

$$\varpi^{k+1}-\varpi^{k}\leq(\frac{(1-\rho_{k})(\alpha_{k}\lambda_{k}-1)}{\rho_{k}}+\gamma_{k+1})\||x^{k+1}-x^{k}\||^{2},\forall k\geq1.\eqno{(5.10)}$$

We claim that
$$\frac{(1-\rho_{k})(\alpha_{k}\lambda_{k}-1)}{\rho_{k}}+\gamma_{k+1}\leq-\theta,\forall k\geq1.\eqno{(5.11)}$$
Let be $k\geq1$. Indeed, by the choice of $\lambda_{k}$, we get
\begin{align*}
&\frac{(1-\rho_{k})(\alpha_{k}\lambda_{k}-1)}{\rho_{k}}+\gamma_{k+1}\leq-\theta\\
\Longleftrightarrow&\rho_{k}(\gamma_{k+1}+\theta)+(\alpha_{k}\lambda_{k}-1)(1-\rho_{k})\leq0,\\
\Longleftrightarrow&\rho_{k}(\gamma_{k+1}+\theta)+\frac{\hat{\delta}\rho_{k}(1-\rho_{k})}{\alpha_{k}+\hat{\delta}\rho_{k}}\leq0,\\
\Longleftrightarrow&(\alpha_{k}+\hat{\delta}\rho_{k})(\gamma_{k+1}+\theta)+\hat{\delta}\rho_{k}\leq\hat{\delta}.
\end{align*}
Thus, by using (5.9), we have
$$(\alpha_{k}+\hat{\delta}\rho_{k})(\gamma_{k+1}+\theta)+\hat{\delta}\rho_{k}\leq(\alpha_{k}+\hat{\delta}\rho_{k})(\alpha(1+\alpha)+\alpha\hat{\delta}+\theta)+\hat{\delta}\rho_{k}\leq\hat{\delta},$$
where the last inequality follows by taking into account the upper
bound considered for $(\rho_{k})_{k\in\mathbb{N}}$ in (ii). Hence
the claim in (5.11) is true.

We obtain from (5.10) and (5.11) that

$$\varpi^{k+1}-\varpi^{k}\leq-\theta\||x^{k+1}-x^{k}\||^{2},\forall k\geq1.\eqno{(5.12)}$$
The sequence $(\varpi_{k})_{k\geq1}$ is nonincreasing and the bound
for $(\alpha_{k})_{k\geq1}$  delivers
$$-\alpha\varphi^{k-1}\leq\varphi^{k}-\alpha\varphi^{k-1}\leq\varpi^{k}\leq\varpi^{1},\forall k\geq1.\eqno{(5.13)}$$
We obtain
$$\varphi^{k}\leq\alpha^{k}\varphi^{0}+\varpi^{1}\sum_{n=0}^{k-1}\alpha^{n}\leq\alpha^{k}\varphi^{0}+\frac{\varpi^{1}}{1-\alpha},\forall k\geq1,$$
where we notice that $\varpi^{1} = \varphi^{1}\geq0$ (due to the
relation $\alpha_{1} = 0$). Combining (5.12) and (5.13), we get for
all $k\geq1$
$$\theta\sum_{n=1}^{k}\||x^{n+1}-x^{n}\||^{2}\leq\varpi^{1}-\varpi^{k+1}\leq\varpi^{1}+\alpha\varphi^{k}\leq\alpha^{k}\varphi^{0}+\frac{\varpi^{1}}{1-\alpha},$$
which shows that
$\sum_{k\in\mathbb{N}}\||x^{k+1}-x^{k}\||^{2}<+\infty$ with respect
to the filtration $(\mathcal {F}^{k})$. By
$w^{k}=x^{k}+\alpha_{k}(x^{k}-x^{k-1})$, (5.9) and Lemma 2.4 we
derive that $\||x^{k}-\tilde{x}\||$ converges with probability one
towards a random variable that is finite almost everywhere.

Given a countable dense subset $\mathrm{Z}$ of $Fix(T)$, there is a
probability one set on which $\||x^{k}-x\||\rightarrow
\mathrm{X}_{x}\in[0,\infty)$ for all $x\in\mathrm{Z}$. Let $x\in
Fix(T)$, let $\varepsilon > 0$, and choose $x\in\mathrm{Z}$ such
that $\||\tilde{x}-x\||\leq\varepsilon$. With probability one, we
have
$$\||x^{k}-\tilde{x}\||\leq\||x^{k}-x\||+\||\tilde{x}-x\||\leq\mathrm{X}_{x}+2\varepsilon,$$
for $k$ large enough. Similarly
$\||x^{k}-\tilde{x}\||\geq\mathrm{X}_{x}-2\varepsilon$, for $k$
large enough. Therefor, we have

$\mathbf{A}_{1}$: There is a probability one set on which
$\||x^{k}-\tilde{x}\||$ converges for every $\tilde{x}\in Fix(T)$.\\
By the definition of $w^{k}$ and the upper bound requested for
 $(\alpha_{k})_{k\geq1}$, we get  there is a probability one set on which
$\||w^{k}-\tilde{x}\||$ converges for every $\tilde{x}\in Fix(T)$.
On the other hand, with the same proof of Theorem 3 of [2], we know
that
\begin{align*}
\mathbb{E}[\||x^{k+1}-\tilde{x}\||^{2}|\mathcal {F}^{k}]&\leq\||w^{k}-\tilde{x}\||^{2}\\
&-\rho_{k}(1-\rho_{k})\|(I-T)w^{k}\|^{2}.\tag{5.14}\\
\end{align*}
From the assumption on $\rho_{k}$, we know that
\begin{align*}
\rho(1-\frac{\hat{\delta}-\alpha[\alpha(1+\alpha)+\alpha\hat{\delta}+\theta]}{\hat{\delta}[1+\alpha(1+\alpha)+\alpha\hat{\delta}+\theta]})\|(I-T)w^{k}\|^{2}&\leq\rho_{k}(1-\rho_{k})\|(I-T)w^{k}\|^{2}\\
&\leq\||w^{k}-\tilde{x}\||^{2}-\mathbb{E}[\||x^{k+1}-\tilde{x}\||^{2}|\mathcal
{F}^{k}]\\
&\leq\alpha_{k}^{2}\||x^{k}-x^{k-1}\||^{2}\\
&-\mathbb{E}[\||x^{k+1}-\tilde{x}\||^{2}|\mathcal
{F}^{k}]\\
&\leq\|x^{k}-x^{k-1}\||^{2}-\mathbb{E}[\||x^{k+1}-\tilde{x}\||^{2}|\mathcal
{F}^{k}].\tag{5.15}
\end{align*}
Taking the expectations on both sides of inequality (5.15) and
iterating over $k$, we obtain
$$\mathbb{E}\|(I-T^{k})w^{k}\|^{2}\leq\frac{1}{\rho(1-\frac{\hat{\delta}-\alpha[\alpha(1+\alpha)+\alpha\hat{\delta}+\theta]}{\hat{\delta}[1+\alpha(1+\alpha)+\alpha\hat{\delta}+\theta]})}(x^{0}-\tilde{x})^{2}.$$
By Markov＊s inequality and Borel Cantelli＊s lemma,we therefore
obtain:

$\mathbf{A}_{2}$: $(I-T)w^{k}\rightarrow0$ almost surely.\\
We now consider an elementary event in the probability one set where
$\mathbf{A}_{1}$ and $\mathbf{A}_{2}$ hold. On this event, since the
sequence $\||w^{k}-\tilde{x}\||$ converges for $\tilde{x}\in
Fix(T)$. the sequence $(w^{k})_{k\in\mathbb{N}}$ is bounded. Since
$T$ is nonexpansive, it is continuous, and $\mathbf{A}_{2}$ shows
that all the accumulation points of  $(w^{k})_{k\in\mathbb{N}}$ are
in $Fix(T)$. It remains to show that these accumulation points
reduce to one point. Assume that $x^{\ast}$ is an accumulation
point. By $\mathbf{A}_{1}$, $\||w^{k}-x^{\ast}\||$ converges.
Therefore, $\lim\||w^{k}-x^{\ast}\||=\liminf\||w^{k}-x^{\ast}\||=0$,
which shows that $x^{\ast}$ is unique.

\end{proof}

\section{Application to stochastic approximation}
\subsection{Problem setting}
Given an integer $N > 1$, consider the problem of minimizing a sum
of composite functions

$$\inf_{x\in\mathcal{X}}\sum_{n=1}^{N}(f_{n}(x)+g_{n}(x)),\eqno{(6.1)}$$
where we make the following assumption:

\begin{ass}
For each $n = 1, ...,N$,\\
 (1) $f_{n}$ is a convex differentiable function on $\mathcal{X}$, and its
gradient $\nabla f_{n}$  be co-coercive w.r.t. a bound, linear,
symmetric and positive linear
maps $\hat{E}^{-1}$;\\
 (2) $g_{n}\in \Gamma_{0}(\mathcal{X})$;\\
 (3) The infimum of Problem (6.1) is attained;\\
 (4) $\cap_{n=1}^{N}ridom g_{n}\neq0.$
\end{ass}

This problem arises for instance in large-scale learning
applications where the learning set is too large to be handled as a
single block. Stochastic minibatch approaches consist in splitting
the data set into $N$ chunks and to process each chunk in some
order, one at a time. The quantity $f_{n}(x) + g_{n}(x)$ measures
the inadequacy between the model (represented by parameter $x$) and
the $n$-th chunk of data. Typically, $f_{n}$ stands for a data
fitting term whereas $g_{n}$ is a regularization term which
penalizes the occurrence of erratic solutions. As an example, the
case where $f_{n}$ is quadratic and $g_{n}$ is the $l_{1}$-norm
reduces to the popular LASSO problem [13]. In particular,  it also
useful to recover sparse signal.

\subsection{ Instantiating the PADMM$^{+}$}
We regard our stochastic minibatch algorithm as an instance of the
PADMM$^{+}$ coupled with a randomized coordinate descent. In order
to end that ,we rephrase Problem (6.1) as
$$\inf_{x\in\mathcal{X}^{N}}\sum_{n=1}^{N}(f_{n}(x)+g_{n}(x))+\iota_{\mathcal{C}}(x),\eqno{(6.2)}$$
where the notation $x_{n}$ represents the $n$-th component of any $x
\in \mathcal{X}^{N}$, $\mathcal{C}$ is the space of vectors $x \in
\mathcal{X}^{N}$ such that $x_{1} = \cdots = x_{N}$. On the space
$\mathcal{X}^{N}$, we set $f(x) = \sum_{ n} f_{n}(x_{n})$, $g(x) =
\sum_{ n} g_{n}(x_{n})$, $h(x) = \iota_{\mathcal{C}}$ and $D =
I_{\mathcal{X}^{N}}$ the identity matrix. Problem (6.2) is
equivalent to
$$\min_{x\in\mathcal{X}^{N}} f(x)+g(x)+ (h\circ D)(x).\eqno{(6.3)}$$
We define the natural scalar product on $\mathcal{X}^{N}$ as
$\langle x,y\rangle=\sum_{n=1}^{N}\langle x_{n},y_{n}\rangle$.
Applying the PADMM$^{+}$ to solve Problem (6.3) leads to the
following iterative scheme:
\begin{align*}
&\eta^{k}=y^{k}+\alpha_{k}(y^{k}-y^{k-1}),\\
&\xi^{k}=x^{k}+\alpha_{k}(x^{k}-x^{k-1}),\\
&z^{k+1}=proj_{\mathcal{C}}(\xi^{k}+\Psi \eta^{k}), \\
&y^{k+1}_{n}=\eta^{k}_{n}+\Psi^{-1}(\xi^{k}_{n}-z^{k+1}_{n} ),\\
&u^{k+1}_{n}=(I-\tilde{T}\Psi^{-1})\xi^{k}_{n}+\tilde{T}\Psi^{-1}z^{k+1}_{n},\\
&x^{k+1}_{n}=\arg\min_{w\in\mathcal {X}}[g_{n}(w)+\langle\nabla
f_{n}(\xi^{k}),w\rangle+\frac{\|w-u^{k+1}_{n}-\tilde{T}y^{k+1}_{n}\|_{\tilde{T}^{-1}}^{2}}{2}],
\end{align*}
where $T$ and $\Psi$ are two diagonal matrices which have same
dimensional, $proj_{\mathcal{C}}$ is the orthogonal projection onto
$\mathcal{C}$. Observe that for any $x \in \mathcal{X}^{N}$,
$proj_{\mathcal{C}}(x)$ is equivalent to $(\bar{x}, \cdots ,
\bar{x})$ where $\bar{x}$ is the average of vector $x$, that is
$\bar{x}=N^{-1}\sum_{n}x_{n}$. Consequently, the components of
$z^{k+1}$ are equal and coincide with $\bar{\xi}^{k}+\Psi
\bar{\eta}^{k}$ where $\bar{\xi}^{k}$ and $\bar{\eta}^{k}$ are the
averages of $\xi^{k}$ and $\eta^{k}$ respectively. By inspecting the
$y^{k}$ $n$-update equation above, we notice that the latter
equality simplifies even further by noting that $\bar{y}^{k+1} = 0$
or, equivalently, $\bar{\eta}^{k} = 0$ for all $k\geq 1$ if the
algorithm is started with $\bar{y}^{0} = 0$. Finally, for any $n$
and $k\geq 1$, the above iterations reduce to
\begin{align*}
&\eta^{k}=y^{k}+\alpha_{k}(y^{k}-y^{k-1}),\\
&\xi^{k}=x^{k}+\alpha_{k}(x^{k}-x^{k-1}),\\
&\bar{\xi}^{k}=\frac{1}{N}\sum_{n=1}^{N} \xi^{k}_{n},\\
&y^{k+1}_{n}=\eta^{k}_{n}+\Psi^{-1}(\xi^{k}_{n}-\bar{\xi}^{k} ),\\
&u^{k+1}_{n}=(I-\tilde{T}\Psi^{-1})\xi^{k}_{n}+\tilde{T}\Psi^{-1}\bar{\xi}^{k},\\
&x^{k+1}_{n}=prox_{\tilde{T}g_{n}}[u^{k+1}_{n}-\tilde{T}(\nabla
f_{n}(\xi^{k}_{n})+y^{k+1}_{n})].
\end{align*}
These iterations can be written more compactly as
\begin{algorithm}[H]
\caption{Minibatch PADMM$^{+}$.}
\begin{algorithmic}\label{1}
\STATE Initialization: Choose $x^{0}, x^{1}\in \mathcal{X}$, $y^{0},
y^{1}\in
\mathcal{Y}$, s.t. $\sum_{n}y^{0}_{n}=0$.\\
Do
$$
\begin{array}{l}
\bullet ~~\eta^{k}=y^{k}+\alpha_{k}(y^{k}-y^{k-1}),\\
~~~~\xi^{k}=x^{k}+\alpha_{k}(x^{k}-x^{k-1}),\\
~~~~\bar{\xi}^{k}=\frac{1}{N}\sum_{n=1}^{N} \xi^{k}_{n},\\
\bullet ~~For~ batches~  n = 1, \cdots ,N,~  do\\
~~~~y^{k+1}_{n}=\eta^{k}_{n}+\Psi^{-1}(\xi^{k}_{n}-\bar{\xi}^{k} ),\\
~~~~x^{k+1}_{n}=prox_{\tilde{T}g_{n}}[(I-2\tilde{T}\Psi^{-1})\xi^{k}_{n}-\tilde{T}\nabla
f_{n}(\xi^{k}_{n})+2\tilde{T}\Psi^{-1}\bar{\xi}^{k}-\tilde{T}\eta^{k}_{n}].\\
\bullet ~~Increment~ k.
\end{array}\eqno{(6.4)}
$$

\end{algorithmic}
\end{algorithm}
The following result is a straightforward consequence of Theorem
4.2.
\begin{thm}
Assume that the minimization Problem (6.3) is consistent,
 $T$ and $\Psi$ are two diagonal matrices which have same
dimensional.  Let Assumption 6.1  hold true and
 $\tilde{T}^{-1}-\frac{1}{2}\hat{E}>0$,
$\|\tilde{T}^{-1}\|-\|\Psi^{-1}\|>\frac{\|\hat{E}\|}{2}$. Let the
sequences $(\bar{x}^{k},y^{k})$ be generated by Minibatch
PADMM$^{+}$. Then for any initial point $( x^{0}, y^{0}), ( x^{1},
y^{1})$ such that $\bar{y}^{0} = 0$, the sequence $\{\bar{x}^{k}\}$
converges to a solution of Problem (6.3).
\end{thm}

At each step $k$, the iterations given above involve the whole set
of functions $f_{n}, g_{n} (n = 1, \cdots,N)$. Our aim is now to
propose an algorithm which involves a single couple of functions
$(f_{n}, g_{n})$ per iteration.
\subsection{ A stochastic minibatch primal-dual splitting  algorithm with  preconditioning} We are now in position to state the main algorithm of this
section. The proposed preconditioned stochastic minibatch
primal-dual splitting algorithm (PSMPDS) is obtained upon applying
the randomized coordinate descent on the minibatch PADMM$^{+}$:

\begin{algorithm}[H]
\caption{PSMPDS.}
\begin{algorithmic}\label{1}
\STATE Initialization: Choose $x^{0}, x^{1}\in \mathcal{X}$, $y^{0},
y^{1}\in
\mathcal{Y}$.\\
Do
$$
\begin{array}{l}
\bullet ~~Define~\eta^{k}=y^{k}+\alpha_{k}(y^{k}-y^{k-1}),\\
~~~~~~~~~~~~~~~\xi^{k}=x^{k}+\alpha_{k}(x^{k}-x^{k-1}),\\
~~~~~~~~~~~~~~~ \bar{\xi}^{k}=\frac{1}{N}\sum_{n=1}^{N} \xi^{k}_{n}, ~ \bar{\eta}^{k}=\frac{1}{N}\sum_{n=1}^{N}\eta^{k}_{n},\\
\bullet ~~Pick ~up ~the ~value ~of  ~\zeta^{k+1}, \\
\bullet ~~For~ batch~n=\zeta^{k+1},~set\\
~~~~y^{k+1}_{n}=\eta^{k}_{n}-\bar{\eta}^{k}+\Psi^{-1}(\xi^{k}_{n}-\bar{\xi}^{k} ),~~~~~~~~~~~~~~~~~~~~~~~~~~~~~~~~~~~~~~~~~~~~~~~~~~(6.5a)\\
~~~~x^{k+1}_{n}=prox_{\tilde{T}g_{n}}[(I-2\tilde{T}\Psi^{-1})\xi^{k}_{n}-\tilde{T}\nabla
f_{n}(\xi^{k}_{n})-\tilde{T}\eta^{k}_{n}+2\tilde{T}(\Psi^{-1}\bar{\xi}^{k}+\bar{\eta}^{k})].(6.5b)\\
\bullet ~~For~ all ~batches ~n\neq\zeta^{k+1},~~ y^{k+1}_{n}=\eta^{k}_{n}, x^{k+1}_{n}=\xi^{k}_{n}.\\
 \bullet ~~Increment~ k.
\end{array}
$$

\end{algorithmic}
\end{algorithm}

\begin{ass}
The random sequence $(\zeta^{k})_{k\in\mathbb{N}^{\ast}}$ is i.i.d.
and satisfies $\mathbb{P}[\zeta^{1} = n] > 0$ for all $n = 1,
...,N$.
\end{ass}

\begin{thm}
Assume that the minimization Problem (6.3) is consistent, $T$ and
$\Psi$ are two diagonal matrices which have same dimensional.  Let
Assumption 6.1 and Assumption 6.2 hold true and
$\tilde{T}^{-1}-\frac{1}{2}\hat{E}>0$,
$\|\tilde{T}^{-1}\|-\|\Psi^{-1}\|>\frac{\|\hat{E}\|}{2}$.
 Then for any initial point $( x^{0}, y^{0}), ( x^{1},
y^{1})$ , the sequence $\{\bar{x}^{k}\}$  generated by PSMPDS
algorithm converges to a solution of Problem (6.3).
\end{thm}
\begin{proof}
Let us define $(\bar{f}, \bar{g}, h, D)=(f, g, h, I_{x^{N}})$ where
the functions $f$, $g$, and $h$ are the ones defined in section 6.2.
If we replace  $T, \Psi$ by $\mu, \tau$, then the iterates
$((y^{k+1}_{ n} )^{N} _{n=1}, (x^{k+1}_{ n} )^{N} _{n=1})$ described
by Equations (6.4) coincide with the iterates $( y^{k+1}, x^{k+1})$
described by Equations (3.10). If we write these equations more
compactly as $(y^{k+1}, x^{k+1}) = T(\xi^{k}, \eta^{k})$ where
$(\xi^{k}, \eta^{k})=(x^{k}, y^{k})+\alpha_{k}[(x^{k},
y^{k})-(x^{k-1}, y^{k-1})]$, and the operator $T$ acts in the space
$\mathcal{Z}=\mathcal{X}^{N}\times\mathcal{X}^{N}$, then from the
proof of Lemma 3.1, we konw that $T$ is $\tilde{a}$-averaged, where
$\tilde{a}=(2-a_{1})^{-1}$ and
$a_{1}=\frac{\|\hat{E}\|}{2}(\|\tilde{T}^{-1}\|-\|\Psi^{-1}\|)^{-1}$
.
 Defining the selection operator $\mathcal{S}_{n }$ on
$\mathcal{Z}$ as $\mathcal{S}_{n }(y, x) = (y_{n}, x_{n})$, we
obtain that $\mathcal{Z} = \mathcal{S}_{1 }(\mathcal{Z})\times
\cdots\times \mathcal{S}_{N }(\mathcal{Z})$ up to an element
reordering. To be compatible with the notations of Section 5.1, we
assume that $J = N$ and that the random sequence $\zeta^{k}$ driving
the PSMPDS algorithm is set valued in
$\{\{1\},\ldots\{N\}\}\subset2^{\mathcal{J}}$. In order to establish
Theorem 6.2, we need to show that the iterates $(y^{k+1}, x^{k+1})$
provided by the PSMPDS algorithm are those who satisfy the equation
$(y^{k+1}, x^{k+1}) = T^{(\zeta^{k+1})}(y^{k}, x^{k})$.
 By the direct application of Theorem 5.1, we can obtain Theorem
 6.2.\\
Let us start with the $y$-update equation. Since $h = \iota_{C}$,
its Legendre-Fenchel transform is $h^{\ast} = \iota_{C^{\perp}}$
where $C^{\perp}$ is the orthogonal complement of $C$ in $\mathcal
{X}^{N}$. Consequently, if we write $(\varsigma^{k+1},
\upsilon^{k+1}) = T(y^{k}, x^{k})$, and replace  $\mu, \tau$  by $T,
\Psi$ then by Eq. (3.10a),

$$\varsigma^{k+1}_{n}=\eta^{k}_{n}-\bar{\eta}^{k}+\Psi^{-1}(\xi^{k}_{n}-\bar{\xi}^{k} )~~n=1,\ldots N.$$
Observe that in general, $\bar{y}^{k}\neq 0$ because in the PSMPDS
algorithm, only one component is updated at a time. If $\{n\} =
\zeta^{k+1}$, then $y^{k+1}_{n} = \varsigma^{k+1}_{n}$ which is Eq.
(6.5a). All other components of $y^{k}$ are carried over to $
y^{k+1}$ .\\
 By Equation (3.10b) we also get
$$\upsilon^{k+1}_{n}=prox_{\tilde{T}g_{n}}[\xi^{k}_{n}-\tilde{T}\nabla
f_{n}(\xi^{k}_{n})-\tilde{T}(2y^{k+1}_{n}-\eta^{k})].$$ If $\{n\} =
\zeta^{k+1}$, then $x^{k+1}_{n}=\upsilon^{k+1}_{n}$ can easily be
shown to be given by (6.5b).

\end{proof}

\section{Distributed optimization}

~~~~Consider a set of $N > 1$ computing agents that cooperate to
solve the minimization Problem (6.1). Here, $f_{n}$, $g_{n}$ are two
private functions available at Agent $n$. Our purpose is to
introduce a random distributed algorithm to solve (6.1). The
algorithm is asynchronous in the sense that some components of the
network are allowed to wake up at random and perform local updates,
while the rest of the network stands still. No coordinator or global
clock is needed. The frequency of activation of the various network
components is likely to vary.

The examples of this problem appear in learning applications where
massive training data sets are distributed over a network and
processed by distinct machines [17], [18], in resource allocation
problems for communication networks [19], or in statistical
estimation problems by sensor networks [20], [21].
\subsection{Network model and problem formulation}
We consider the network as a graph $G = (Q,E)$ where $Q = \{1,
\cdots ,N\}$ is the set of agents/nodes and $E\subset \{1, \cdots
,N\}^{2}$ is the set of undirected edges. We write $n\sim m$
whenever ${n,m}\in E$. Practically, $n\sim m $ means that agents $n$
and $m$ can communicate with each other.

\begin{ass}
$G$ is connected and has no self loop.
\end{ass}

Now we introduce some notations. For any $x\in\mathcal{X}^{|Q|}$, we
denote by $x_{n}$ the components of $x$, i.e., $x = (x_{n})_{n\in
Q}$. We regard the functions $f$ and $g$ on
$\mathcal{X}^{|Q|}\rightarrow(-\infty,+\infty]$ as $f(x)=\sum_{n\in
Q}f_{n}(x_{n})$ and $g(x)=\sum_{n\in Q}g_{n}(x_{n})$. So the Problem
(6.1) is equal to the minimization of $f(x)+g(x)$ under the
constraint that all components of $x$ are equal.

Next we write the latter constraint in a way that involves the graph
$G$. We replace the global consensus constraint by a modified
version of the function $\iota_{\mathcal{C}}$ . The purpose of us is
to ensure global consensus through local consensus over every edge
of the graph.

For any $\epsilon\in E$, say $\epsilon= \{n,m\}\in Q$ , we define
the linear operator $D_{\epsilon}(x) : \mathcal{X}^{|Q|} \rightarrow
\mathcal{X}^{2}$ as $D_{\epsilon}(x) = (x_{n}, x_{m})$ where we
assume some ordering on the nodes to avoid any ambiguity on the
definition of $D$. We construct the linear operator
$D:\mathcal{X}^{|Q|} \rightarrow\mathcal{Y}\triangleq
\mathcal{X}^{2|Q|}$ as $D(x)=(D_{\epsilon}(x))_{\epsilon\in E}$
where we also assume some ordering on the edges. Any vector $y \in
\mathcal{Y}$ will be written as $y = (y_{\epsilon})_{\epsilon\in E}$
where, writing $\epsilon= \{n,m\} \in E$, the component
$y_{\epsilon}$ will be represented by the couple $y_{\epsilon}=
(y_{\epsilon}(n), y_{\epsilon}(m))$ with $n < m$. We also introduce
the subspace of
 $\mathcal{X}^{2}$ defined as $\mathcal{C}_{2} = \{(x, x) : x \in \mathcal{X}\}$. Finally, we define $h : \mathcal{Y} \rightarrow (
-\infty,+\infty]$ as

$$h(y)=\sum_{\epsilon\in E}\iota_{\mathcal{C}_{2}}(y_{\epsilon}).\eqno{(7.1)}$$
Then we consider the following problem:
$$\min_{x\in\mathcal{X}^{|Q|}} f(x)+g(x)+ (h\circ D)(x).\eqno{(7.2)}$$

\begin{lem}
([2]). Let Assumptions 7.1 hold true.  The minimizers of (7.2) are
the tuples $(x^{\ast}, \cdots , x^{\ast})$ where $x^{\ast}$ is any
minimizer of (6.1).
\end{lem}

\subsection{Instantiating the PADMM$^{+}$}

Now we use the PADMM$^{+}$ to solve the Problem (7.2). Since the
newly defined function $h$ is separable with respect to the
$(y_{\epsilon})_{\epsilon\in E}$, we get

$$prox_{\tilde{T} h}(y)=(prox_{\tilde{T} \iota_{\mathcal{C}_{2}}}(y_{\epsilon}))_{\epsilon\in E}=((\bar{y}_{\epsilon}, \bar{y}_{\epsilon}))_{\epsilon\in E},$$
where $\bar{y}_{\epsilon}=(y_{\epsilon}(n)+y_{\epsilon}(m))/2$ if
$\epsilon=\{n,m\}$. With this at hand, the update equation (4.1a) of
the PADMM$^{+}$ can be  written as
$$z^{k+1}=((\bar{z}_{\epsilon}^{k+1},\bar{z}_{\epsilon}^{k+1}))_{\epsilon\in E},$$
where
$$\bar{z}^{k+1}=\frac{\xi^{k}_{n}+\xi^{k}_{m}}{2}+\frac{\Psi(\eta_{\epsilon}^{k}(n)+\eta_{\epsilon}^{k}(m))}{2},$$
 for any $\epsilon= \{n,m\}\in E$.
  Plugging this equality into Eq. (4.1b) of the
PADMM$^{+}$, it can be seen that
$\eta_{\epsilon}^{k}(n)=-\eta_{\epsilon}^{k}(m)$. Therefore
$$\bar{z}^{k+1}=\frac{\xi^{k}_{n}+\xi^{k}_{m}}{2},$$
for any $k \geq 1$. Moreover
$$y_{\epsilon}^{k+1}=\frac{\Psi^{-1}(\xi^{k}_{n}-\xi^{k}_{m})}{2}+\eta_{\epsilon}^{k}(n).$$
Observe that the $n$-th component of the vector $D^{\ast}Dx$
coincides with $ d_{n}x_{n}$, where $ d_{n}$ is the degree (i.e.,
the number of neighbors) of node $n$. From (4.1d) of the
PADMM$^{+}$, the $n^{th}$ component of  $  x^{k+1}$ can be written
$$x^{k+1}_{n}=prox_{\tilde{T}g_{n}/d_{n}}[\frac{(D^{\ast}(u^{k+1}-\tilde{T}y^{k+1}))_{n}-\tilde{T}\nabla
f_{n}(\xi^{k}_{n})}{d_{n}}],$$ where for any $y\in\mathcal{Y}$,
$$(D^{\ast}y)_{n}=\sum_{m:\{n,m\}\in E}y_{\{n,m\}}(n)$$
is the $n$-th component of $D^{\ast}y \in \mathcal{X}^{|Q |}$.
Plugging Eq.  (4.1c) of the PADMM$^{+}$ together with the
expressions of $\bar{z}^{k+1}_{\{n,m\}}$ and $y^{k+1}_{\{n,m\}}$ in
the argument of $prox_{\tilde{T}g_{n}/d_{n}}$ , we can have
\begin{align*}
x^{k+1}_{n}&=prox_{\tilde{T}g_{n}/d_{n}}[(I-\tilde{T}\Psi^{-1})\xi^{k}_{n}-\frac{\tilde{T}}{d_{n}}\nabla
f_{n}(\xi^{k}_{n})+\frac{\tilde{T}}{d_{n}}\sum_{m:\{n,m\}\in
E}(\Psi^{-1}\xi^{k}_{m}-\eta^{k}_{\{n,m\}}(n))].
\end{align*}
The algorithm is finally described by the following procedure: Prior
to the clock tick $k + 1$, the node $n$ has in its memory the
variables $x^{k}_{n}$, $\{y_{\{n,m\}}^{k}(n)\}_{m\thicksim n}$,
 and  $\{x^{k}_{m}\}_{m\thicksim n}$.
\begin{algorithm}[H]
\caption{Distributed PADMM$^{+}$.}
\begin{algorithmic}\label{1}
\STATE Initialization: Choose $x^{0}, x^{1}\in \mathcal{X}$, $y^{0},
y^{1}\in
\mathcal{Y}$, s.t. $\sum_{n}y^{0}_{n}=0$.\\
Do
$$
\begin{array}{l}
\bullet ~~Define~\eta^{k}=y^{k}+\alpha_{k}(y^{k}-y^{k-1}),\\
~~~~~~~~~~~~~~~\xi^{k}=x^{k}+\alpha_{k}(x^{k}-x^{k-1}),\\
\bullet ~~For~ any~ n \in Q ,~ Agent ~n ~performs ~the~ following~
operations: \\
~~~~y_{\{n,m\}}^{k+1}(n)=\eta_{\{n,m\}}^{k}(n)+\frac{\xi^{k}_{n}-\xi^{k}_{m}}{2},~~for~all~m\thicksim n, ~~~~~~~~~~~~~~~~~~~~~~~~~~~~~~~~~~(7.3a)\\
~~~~~~~~~~x^{k+1}_{n}=prox_{\tilde{T}g_{n}/d_{n}}[(I-\tilde{T}\Psi^{-1})\xi^{k}_{n}-\frac{\tilde{T}}{d_{n}}\nabla
f_{n}(\xi^{k}_{n})\\
~~~~~~~~~~~~~~~~~~~+\frac{\tilde{T}}{d_{n}}\sum_{m:\{n,m\}\in
E}(\Psi^{-1}\xi^{k}_{m}-\eta^{k}_{\{n,m\}}(n))].~~~~~~~~~~~~~~~~~~~~~~~~~~~~~~~(7.3b)\\
\bullet ~~Agent~ n ~sends~ the~ parameter~ y^{k+1}_{n}, x^{k+1}_{n}
 ~to~ their~ neighbors~ respectively.\\
 \bullet ~~Increment~ k.
\end{array}
$$

\end{algorithmic}
\end{algorithm}

\begin{thm}
 Assume that the minimization Problem (6.1) is consistent,
 $T$ and
$\Psi$ are two diagonal matrices which have same dimensional.  Let
Assumption 6.1 and Assumption 7.1 hold true and
$\tilde{T}^{-1}-\frac{1}{2}\hat{E}>0$,
$\|\tilde{T}^{-1}\|-\|\Psi^{-1}\|>\frac{\|\hat{E}\|}{2}$. Let
$(x^{k})_{k\in\mathbb{N}} $ be the sequence generated by Distributed
PADMM$^{+}$ for any initial point $( x^{0}, y^{0})$,  $( x^{1},
y^{1})$. Then for all $n \in Q$ the sequence $(x^{k}_{n})_{k\in
\mathbb{N}}$ converges to a solution of Problem (6.1).
\end{thm}

\subsection{A Distributed asynchronous primal-dual splitting  algorithm with  preconditioning} In this section, we use the randomized coordinate descent
on the above algorithm, we call this algorithm as preconditioned
distributed asynchronous primal-dual splitting  algorithm
(PDAPDS). This algorithm has the following attractive property:\\
Firstly, it significantly accelerates the convergence on problems
with irregular $D$. Moreover, it leaves the computational complexity
of the iterations basically unchanged. Finally, if we let
$(\zeta^{k})_{k\in \mathbb{N}}$ be a sequence of i.i.d. random
variables valued in $2^{ Q}$. The value taken by $\zeta^{k}$
represents the agents that will be activated and perform a prox on
their $x$ variable at moment $k$. The asynchronous algorithm goes as
follows:

\begin{algorithm}[H]
\caption{PDAPDS.}
\begin{algorithmic}\label{1}
\STATE Initialization: $x^{0},x^{1}\in \mathcal{X}$, $y^{0},y^{1}\in
\mathcal{Y}$.\\
Do
$$
\begin{array}{l}
\bullet ~~Define~\eta^{k}=y^{k}+\alpha_{k}(y^{k}-y^{k-1}),\\
~~~~~~~~~~~~~~~\xi^{k}=x^{k}+\alpha_{k}(x^{k}-x^{k-1}),\\
\bullet ~~Select~ a~ random~ set ~of~ agents~ \zeta^{k+1} =\mathcal{B}. \\
\bullet ~~For~ any~ n \in \mathcal{B},~ Agent~ n~ performs~ the
~following~
operations:\\
~~~~-For ~all~ m \thicksim n, do\\
~~~~~~~~y_{\{n,m\}}^{k+1}(n)=\frac{\eta_{\{n,m\}}^{k}(n)-\eta_{\{n,m\}}^{k}(m)}{2}+\frac{\xi^{k}_{n}-\xi^{k}_{m}}{2},\\
~~~~-x^{k+1}_{n}=prox_{\tilde{T}g_{n}/d_{n}}[(I-\tilde{T}\Psi^{-1})\xi^{k}_{n}-\frac{\tilde{T}}{d_{n}}\nabla
f_{n}(\xi^{k}_{n})\\
~~~~~~~~~~~~~~~~~+\frac{\tilde{T}}{d_{n}}\sum_{m:\{n\thicksim m\}\in
E}(\Psi^{-1}\xi^{k}_{m}+\eta^{k}_{\{n,m\}}(m))].\\
~~~~-For ~all~ m \thicksim n,~ send
\{x^{k+1}_{n},y_{\{n,m\}}^{k+1}(n)\}~to~ Neighbor~
m.\\
\bullet ~~For ~any~ agent ~n\neq\mathcal{B}, ~
x^{k+1}_{n}=\xi^{k}_{n},
~and~~y_{\{n,m\}}^{k+1}(n)=\eta_{\{n,m\}}^{k}(n)\\
~~~~for~ all~ m \thicksim n.\\
 \bullet ~~Increment~ k.
\end{array}
$$

\end{algorithmic}
\end{algorithm}

\begin{ass}
The collections of sets $\{\mathcal{B}_{1},\mathcal{B}_{2},\ldots\}$
such that $\mathbb{P}[\zeta^{1} = \mathcal{B}_{i}]$ is positive
satisfies $ \bigcup\mathcal{B}_{i} =Q$.
\end{ass}

\begin{thm}
  Assume that the minimization Problem (6.1) is consistent,
$T$ and $\Psi$ are two diagonal matrices which have same
dimensional.  Let Assumption 6.1, Assumption 7.1 and 7.2 hold true,
and $\tilde{T}^{-1}-\frac{1}{2}\hat{E}>0$,
$\|\tilde{T}^{-1}\|-\|\Psi^{-1}\|>\frac{\|\hat{E}\|}{2}$. Let
$(x^{k}_{n})_{n\in Q}$ be the sequence generated by PDAPDS for any
initial point $( x^{0}, y^{0})$,  $( x^{1}, y^{1})$. Then the
sequence $x^{k}_{1},\ldots,x^{k}_{|Q|}$ converges to a solution of
Problem (6.1).
\end{thm}

\begin{proof}
Let  $(\bar{f}, \bar{g}, h)=(f\circ D^{-1}, g\circ D^{-1}, h)$ where
$f, g, h$  and $D$  are the ones defined in the Problem 7.2. By
Equations (3.10). We write these equations more compactly as
$(y^{k+1}, x^{k+1}) = T(\xi^{k}, \eta^{k})$ where $(\xi^{k},
\eta^{k})=(x^{k}, y^{k})+\alpha_{k}[(x^{k}, y^{k})-(x^{k-1},
y^{k-1})]$ , the operator $T$ acts in the space
$\mathcal{Z}=\mathcal {Y}\times\mathcal{R}$, and $\mathcal{R}$ is
the image of $\mathcal{X}^{|Q|}$ by $D$. Then from the proof of
Lemma 3.1, we konw that $T$ is $\tilde{a}$-averaged, where
$\tilde{a}=(2-a_{1})^{-1}$ and
$a_{1}=\frac{\|\hat{E}\|}{2}(\|\tilde{T}^{-1}\|-\|\Psi^{-1}\|)^{-1}$.
Defining the selection operator $\mathcal{S}_{n }$ on $\mathcal{Z}$
as $\mathcal{S}_{n }(\eta, D\xi) = (\eta_{\epsilon}(n)_{\epsilon\in
Q:n\in\epsilon}, \xi_{n})$. So, we obtain that $\mathcal{Z} =
\mathcal{S}_{1 }(\mathcal{Z})\times \cdots\times \mathcal{S}_{|Q|
}(\mathcal{Z})$ up to an element reordering. Identifying the set
$\mathcal{J}$ introduced in the notations of Section 5.1 with $Q$,
the operator $T^{(\zeta^{k})}$ is defined as follows:

$$
\mathcal{S}_{n }(T^{(\zeta^{k})}(\eta, D\xi)) = \left\{
\begin{array}{l}
\mathcal{S}_{n }(T(\eta, D\xi)),\,\,\, \,\,if \,n\,
\in\zeta^{k},\\
\mathcal{S}_{n }(\eta, D\xi),\,\,\,\,\,\,\,\,\,\,\, \,\,if \,n\,
\neq\zeta^{k} .
\end{array}
\right.
$$
Then by Theorem 5.1, we know the sequence $(y^{k+1}, Dx^{k+1}) =
T^{(\zeta^{k+1})}(\eta^{k}, D\xi^{k})$ converges almost surely to a
solution of Problem (3.11). Moreover, from Lemma 7.1, we have the
sequence $x^{k}$ converges almost surely to a solution of Problem
(6.1).\\
Therefore we need to show that the operator $T^{(\zeta^{k+1})}$ is
translated into the PDAPDS algorithm. The definition (7.1) of  $h$
shows that
$$h^{\ast}(\varphi)=\Sigma_{\epsilon\in E}\iota_{\mathcal
{C}_{2}^{\perp}}(\varphi_{\epsilon}),$$ where
$\mathcal{C}_{2}^{\perp}= \{(x,-x) : x\in\mathcal {X}\}$. Therefore,
writing
$$(\varsigma^{k+1}, \upsilon^{k+1}=Dq^{k+1}) = T(\eta^{k},
\lambda^{k}=D\xi^{k}),$$  then by Eq. (3.10a),
$$\varsigma_{\epsilon}^{k+1}=proj_{\mathcal{C}_{2}^{\perp}}(\eta^{k}_{\epsilon}+\Psi^{-1}\lambda^{k}_{\epsilon}).$$
Observe that contrary to the case of the synchronous algorithm
(7.3), there is no reason here for which
$proj_{\mathcal{C}_{2}^{\perp}}(\eta^{k}_{\epsilon})= 0$.  Getting
back to $(y^{k+1},D x^{k+1}) = T^{(\zeta^{k+1})}(\eta^{k},
\lambda^{k}=D\xi^{k})$, we have for all $n \in \zeta^{k+1}$ and all
$m \thicksim n$,
\begin{align*}
y^{k+1}_{\{n,m\}}(n) &=
\frac{\eta^{k}_{\{n,m\}}(n)-\eta^{k}_{\{n,m\}}(m)}{2}+\frac{\lambda^{k}_{\{n,m\}}(n)-\lambda^{k}_{\{n,m\}}(m)}{2}\\
&=\frac{\eta^{k}_{\{n,m\}}(n)-\eta^{k}_{\{n,m\}}(m)}{2}+\frac{\xi^{k}_{n}-\xi^{k}_{m}}{2}.
\end{align*}
By Equation (3.10b) we also get
$$\upsilon^{k+1}=\arg\min_{w\in\mathcal
{R}}[\bar{g}(w)+\langle\nabla
\bar{f}(\lambda^{k}),w\rangle+\frac{\|w-\lambda^{k}+\tilde{T}(2y^{k+1}-\eta^{k})\|_{\tilde{T}^{-1}}^{2}}{2}].$$
Upon noting that $\bar{g}(D\xi) = g(\xi)$ and $\langle\nabla
\bar{f}(\lambda^{k}),D\xi\rangle= \langle(D^{-1})^{\ast}\nabla
f(D^{-1}D\xi^{k}),D\xi\rangle = \langle\nabla f(\xi^{k}),
\xi\rangle$, the above equation becomes
$$q^{k+1}=\arg\min_{w\in\mathcal
{X}}[g(w)+\langle\nabla
f(\xi^{k}),w\rangle+\frac{\|D(w-\xi^{k})+\tilde{T}(2y^{k+1}-\eta^{k})\|_{\tilde{T}^{-1}}^{2}}{2}].$$
Recall that $(D^{\ast}Dx)_{n} = d_{n}x_{n}$. Hence, for all $n\in
\zeta^{k+1}$, we get after some computations
$$x^{k+1}_{n}=prox_{\tilde{T}g_{n}/d_{n}}[\xi^{k}_{n}-\frac{\tilde{T}}{d_{n}}\nabla
f_{n}(\xi^{k}_{n})+\frac{\tilde{T}}{d_{n}}(D^{\ast}(2y^{k+1}-\eta^{k}))_{n}].$$
 Using the identity $(D^{\ast}y)_{n}=\sum_{m:\{n,m\}\in
E}y_{\{n,m\}}(n)$ , it can easy check these equations coincides with
the $x$-update  in the PDAPDS algorithm.

\end{proof}

\section{Numerical experiments}
We consider the problem of $l_{1}$-regularized logistic regression.
Denoting by $m$ the number of observations and by $q$ the number of
features, the optimization problem writes
$$\inf_{x\in \mathbb{R}^{q}}\frac{1}{m}\sum_{i=1}^{m}\log(1+e^{-y_{i}a_{i}^{T}x})+\tau\|x\|_{1},\eqno{(8.1)}$$
where the $(y_{i})_{i=1}^{m}$ are in $\{-1,+1\}$, the
$(a_{i})_{i=1}^{m}$ are in $\mathbb{R}^{q}$, and $\tau>0$ is a
scalar. Let $(\mathcal{W})_{n=1}^{N}$ indicate a partition of $\{1,
. . . ,m\}$. The optimization problem then writes

$$\inf_{x\in \mathbb{R}^{q}}\sum_{n=1}^{N}\sum_{i\in\mathcal{W}_{n}}\frac{1}{m}\log(1+e^{-y_{i}a_{i}^{T}x})+\tau\|x\|_{1},\eqno{(8.2)}$$

or, splitting the problem between the batches

$$\inf_{x\in \mathbb{R}^{N^{q}}}\sum_{n=1}^{N}(\sum_{i\in\mathcal{W}_{n}}\frac{1}{m}\log(1+e^{-y_{i}a_{i}^{T}x_{n}})+\frac{\tau}{N}\|x_{n}\|_{1})+\iota_{\mathcal{C}(x)},\eqno{(8.3)}$$
where $x = (x_{1}, ..., x_{N})$ is in  $\mathbb{R}^{N^{q}}$. It is
easy to see that Problems (8.1), (8.2) and (8.3) are equivalent and
Problem (8.3) is in the form of (6.2).
\section{Conclusion}

In  this paper, we introduced a new framework for stochastic
coordinate descent and used on a algorithm called ADMMDS$^{+}$. As a
byproduct, we obtained a stochastic approximation algorithm with
dynamic stepsize which can be used to handle distinct data blocks
sequentially. We also obtained an asynchronous distributed algorithm
with dynamic stepsize which enables the processing of distinct
blocks on different machines.

\noindent \textbf{Acknowledgements}

This work was supported by the National Natural Science Foundation
of China (11131006, 41390450, 91330204, 11401293), the National
Basic Research Program of China (2013CB 329404), the Natural Science
Foundations of Jiangxi Province (CA20110\\
7114, 20114BAB 201004).

\end{document}